\documentclass[abstract=true,10pt]{scrartcl}

\RequirePackage{amsthm,amsmath,amsfonts,amssymb}
\RequirePackage{graphicx}

\usepackage[left=2cm,right=2.5cm]{geometry}
\usepackage{authblk}
\usepackage[table,xcdraw]{xcolor}
\usepackage[english]{babel}
\usepackage[latin1,utf8]{inputenx}
\usepackage{alltt}
\usepackage{amssymb}
\usepackage{latexsym}
\usepackage{lipsum}
\usepackage{graphicx}
\usepackage{mathrsfs}
\usepackage{epsfig}
\usepackage{hyperref}
\usepackage{adjustbox}
\usepackage[hypcap]{caption}

\usepackage{aligned-overset}
\usepackage{pdfpages}
\usepackage{amsthm}
\usepackage{amsmath}
\usepackage{float}
\usepackage{bbm}
\usepackage{subcaption}

\allowdisplaybreaks

\newcommand{\C}{\mathcal{C}} 
\newcommand{\N}{\mathbb{N}} 
\newcommand{\Sd}{S_\delta} 
\newcommand{\DSL}{\mathcal{D}^{LSL}} 
\newcommand{\CSL}{\mathcal{C}^{LSL}} 

\newcommand{\unit}{[0,1]} 
\newcommand{\borelftwo}{\mathcal{B}(\unit^2)}  
\newcommand{\borelf}{\mathcal{B}(\unit)}  
 
\newtheorem{Theorem}{Theorem}[section]
\newtheorem{Definition}[Theorem]{Definition} 
\newtheorem{Lemma}[Theorem]{Lemma}	
\theoremstyle{definition}
\newtheorem{example}[Theorem]{Example}

\theoremstyle{definition}
\newtheorem{remark}[Theorem]{Remark} 

\theoremstyle{remark}

\providecommand{\keywords}[1]
{
  \small	
  \textbf{\textit{Keywords }} #1
}

\makeatletter
\renewcommand{\maketitle}{\bgroup\setlength{\parindent}{0pt}
\begin{center}
  \textbf{\@title}
\end{center}

\begin{flushleft}
	\@author
\end{flushleft}\egroup
}
\makeatother

\title{\begin{LARGE}On bivariate lower semilinear copulas and the star product\end{LARGE}}

\date{}

\author[1,a]{Lea Maislinger}
\author[1,b]{Wolfgang Trutschnig} 

\affil[1]{\begin{small} Department of Artificial Intelligence and Human Interfaces, University of Salzburg, Austria \end{small} \bigskip}

\affil[a]{\begin{small}\url{lea.maislinger@plus.ac.at}\end{small}}
\affil[b]{\begin{small}\url{wolfgang@trutschnig.net}\end{small}}

\begin{document}

\maketitle


\begin{abstract}
\noindent We revisit the family $\mathcal{C}^{LSL}$ of all bivariate lower semilinear 
(LSL) copulas first introduced by Durante et al. in 2008 and, using the 
characterization of LSL copulas in terms of diagonals with specific properties, derive 
several novel and partially unexpected results. In particular we prove that the star 
product (also known as Markov product) $S_{\delta_1}*S_{\delta_2}$ of two LSL copulas 
$S_{\delta_1},S_{\delta_2}$ is again a LSL copula, i.e., that the family 
$\mathcal{C}^{LSL}$ is closed with respect to the star product. Moreover, we show that 
translating the star product to the class of corresponding diagonals 
$\mathcal{D}^{LSL}$ allows to determine the limit of the sequence  $S_\delta, 
S_\delta*S_\delta, S_\delta*S_\delta*S_\delta,\ldots$ for every diagonal $\delta \in 
\mathcal{D}^{LSL}$. In fact, for every LSL copula $S_\delta$ the sequence 
$(S_\delta^{*n})_{n \in \mathbb{N}}$ converges to some LSL copula 
$S_{\overline{\delta}}$, the limit $S_{\overline{\delta}}$ is idempotent, and the class 
of all idempotent LSL copulas allows for a simple characterization.\\
Complementing these results we then focus on concordance of LSL copulas. After deriving 
simple formulas for Kendall's $\tau$ and Spearman's $\rho$ we study the exact region 
$\Omega^{LSL}$ determined by these two concordance measures of all elements in 
$\mathcal{C}^{LSL}$, derive a sharp lower bound and finally show that $\Omega^{LSL}$ is 
convex and compact.
\end{abstract}

\keywords{Copula, Lower semilinear copula, Markov product, concordance, Markov kernel}

\section{Introduction}

    Considering that the class of bivariate copulas is quite diverse (in the sense of containing both, very regular/smooth 
    elements as well as distributions with fractal support, see \cite{TruFS}) it seems natural to look for 
    subclasses which, on the one hand, are handy and well understood and, on the other hand, are 
    sufficiently large to be relevant for applications. 
    Extreme Value copulas as well as Archimedean copulas are standard classes fulfilling these properties. 
    Both of them are characterized via a univariate function or, equivalently a probability measure on $[0,1]$ or $[0,\infty)$
    respectively, see, e.g., \cite{KTFSTW} and the references therein. \\	
	Asking for linearity along some segments of the unit square (and the resul\-ting simple analytic forms)
	in 2008 Durante et al. (see \cite{DuranteSLC}) introduced and analyzed 
	the family of so-called bivariate lower (and upper) semilinear co\-pulas. 
	The authors provided (among various other results) a nice stochastic interpretation and 
	characterized lower semilinear copulas (LSL copulas, for short) 
	in terms of another class of univariate functions: the class $\mathcal{D}^{LSL}$ of (copula) diagonals 
	with some additional growth conditions, defined by
	\begin{align}\label{DLSL}
	\DSL := \left\lbrace \delta \in \mathcal{D}: \varphi_\delta 
 	\text{ non-decreasing}, \eta_\delta \text{ non-increasing}
 	\right\rbrace.
    \end{align}
 	Thereby $\mathcal{D}$ denotes the family of all diagonals of bivariate copulas and the functions
 	$\varphi_\delta,\eta_\delta:(0,1] \rightarrow [0,\infty)$ are given by
 	$$
 	\varphi_\delta(x):= \frac{\delta(x)}{x},\quad \eta_\delta(x):= \frac{\delta(x)}{x^2}.
 	$$ 
 	For every fixed $\delta \in \DSL$ the corresponding LSL copula $S_\delta$ is then given by  
    \begin{align}\label{semcopula1}
 	 S_\delta(x,y) := \begin{cases}
 	 					y \, \frac{\delta(x)}{x} & \text{if } y\leq x,\\
 	 					x \, \frac{\delta(y)}{y} & \text{otherwise.}
 	 					\end{cases}
 	 \end{align}	
 	Obviously $S_\delta$ is symmetric, i.e., we have $S_\delta(x,y)=S_\delta(y,x)=:S_\delta^t(x,y)$ for all $x,y \in [0,1]$.  
    Letting $\CSL$ denote the family of all bivariate LSL copula it is straightforward to verify that
    $\CSL$ is convex and that $(\CSL,d_\infty)$ is compact. As a consequence, in \cite{ExtremeSemilin} Durante et al.
    revisited $\CSL$ and provided a nice cha\-rac\-terization of extreme points (in the Krein-Milman sense) of $\CSL$.\\ 	
 	Multivariate extensions of semilinear copulas were studied by Arias-García et al. in \cite{Multivar}
 	and further analyzed and provided with a probabilistic interpretation by Sloot and Scherer in \cite{SlootDiss}. 
	In the latter paper the authors in particular showed that in the multivariate setting upper 
	semilinear copulas are quite pathological in the sense that they concentrate their mass on finitely many 
	hyperplanes (and hence are singular).
	
	Here we revisit the class $\CSL$ of bivariate LSL copulas and show that, although the class is well studied 
	and easy to handle, it possesses additional properties that are at least partially surprising. 
    The latter in particular applies to the behavior of LSL copulas in connection with the so-called star product 
    (also known as Markov product) of copulas. The star product was introduced by Darsow et el. in 1992 (see \cite{DarsowCMP}) 
    and has since been studied in numerous papers. In 1996 Olsen et al. (see \cite{Ol}) 
    showed that the family $\mathcal{C}$
    of all bivariate copulas with the star product as binary operation and the space $(\mathcal{M}, \circ)$ of Markov 
    operators with the composition as binary operation are isomorphic - a result implying that
    studying the star product of copulas can as well be done by studying the corresponding Markov operators.
    Another rational, why the name Markov product is perhaps more adequate than the name star product was provided in 
    \cite{TruFS}, where the authors showed that $A*B$ just corresponds to the standard composition of the Markov kernels
    $K_A$ and $K_B$ (transition probabilities) well known in the context of Markov chains. \\    
    Apart from the family of completely dependent copulas and the family of checkerboard copulas (see, e.g., 
    \cite{KTFSTW} and the references therein) only fully parametric classes (like the Fr\'echet class, see \cite{NelsenCopulas},
    and Gauss copulas, see \cite{FuchsMO})
    are known to be closed under the star product. Considering the lower Fr\'echet Hoeffding bound $W$ we have 
    $W*W=M$, so although $W$ is an Archimedian $W*W=M$ is not. In other words: the family of Archimedean copulas is 
    not closed w.r.t. the star product. For the class of Extreme-Value copulas more tedious calculations 
    yield the same result, the class is not closed w.r.t. the star product either. \\
    It is therefore quite surprising that the star product $S_{\delta_1}*S_{\delta_2}$ of two 
    LSL copulas $S_{\delta_1}, S_{\delta_2} \in \CSL$ is again a LSL copula.     
    Building upon this fact allows to translate the star product to the class $\DSL$ of diagonals and to study
    the limit behavior of iterates of the star product $S_\delta^{*n}=S_\delta * S_\delta* \cdots * S_\delta$
    directly in terms of the limit behavior of the sequence of corresponding diagonals $(\delta^{*n})_{n \in \mathbb{N}}$.
    As one of the main results of this contribution 
    we will show that for every $\delta \in \DSL$ the sequence $(\delta^{*n})_{n \in \mathbb{N}}$
    converges uniformly to some $\overline{\delta} \in \DSL$ and that $\overline{\delta}$ is a fixed point of the star
    product in the sense that $\overline{\delta}*\overline{\delta}=\overline{\delta} $ holds. 
    Translating back to the class $\CSL$: for every $S_\delta \in \CSL$ there exists some $S_{\overline{\delta}} \in \CSL$
    such that 
    $$
    \lim_{n \rightarrow \infty} d_\infty(S_\delta^{*n},S_{\overline{\delta}})=0
    $$ 
    holds, and the limit copula $S_{\overline{\delta}}$ is idempotent, i.e.,  
    $S_{\overline{\delta}}*S_{\overline{\delta}}=S_{\overline{\delta}}$ holds. 
    Roun\-ding off these findings we provide a simple characterization of 
    all idempotent LSL copulas. 
     
    In the second part of the paper we focus on concordance of LSL copulas and study in particular 
    Kendall's $\tau$ and Spearman's $\rho$. 
    After deriving simple expressions for both concordance measures we study the $\tau$-$\rho$-region 
    $\Omega^{LSL}$, defined by
    $$
    \Omega^{LSL}:=\left\{(\tau(S_\delta),\rho(S_\delta)): \, S_\delta \in \CSL \right\}.
    $$
    We conjecture that $\Omega^{LSL}$ coincides with the set  
    \begin{equation*}\label{ew:lower.upper}
    R:=\left\{(x,y) \in [0,1]^2: x \leq y \leq 1 - (1-x)^{\frac{3}{2}}  \right\},
    \end{equation*}
    we were, however, only able to prove the lower inequality and to show that it is sharp. Proving the upper
    inequality (arising from running numerous simulations) remains an open problem. 
    Finally, despite not knowing the exact upper bound we show that $\Omega^{LSL}$ is compact and convex.  \\  
    	
	The rest of this paper is organized as follows:
	Section 2 gathers some notations and preliminaries and recalls some facts about LSL copulas mainly going back 
	to \cite{DuranteSLC}. Section 3 derives the Markov kernels of LSL copulas and considers two parametric 
	classes which will prove important in the sequel. 
	Section 4 starts with proving the fact that the star product of two LSL copulas is again a LSL copula and 
	then derives the afore-mentioned results on the limit of star product iterates of LSL copulas and their (idempotent)
	limits. Finally, Section 5 is devoted to studying Kendall's $\tau$ and Spearman's $\rho$ and their interplay in the 
	family $\CSL$. Several examples and graphics illustrate the studied procedures and some underlying ideas.
	

\section{Notation and preliminaries}
	
	\noindent For every metric space $(S,d)$ the Borel $\sigma$-field on $S$ will 
	be denoted by $\mathcal{B}(S)$. The two-dimensional Lebesgue measure on $\mathcal{B}([0,1]^2)$ will be denoted by $\lambda_2$, 
	the one-dimensional Lebesgue measure by $\lambda$.
	
    In the sequel we will write $\C$ for the family of all bivariate copulas, $\Pi$ denotes the product copula,
    $M$ the minimum copula and $W$ the lower Fr\'echet Hoeffding bound. $C^t$ will denote the transpose of $C$, i.e., 
    the copula fulfilling $C^t(x,y)=C(y,x)$. 	
	For every $C \in \C$ the corresponding doubly 
	stochastic measure will be denoted by $\mu_C$, i.e., $\mu_C([0,x]\times[0,y]) = C(x,y)$ 
	for all $x,y\in \unit$ (and $\mu_C$ is extended to full $\borelftwo$ in the standard measure-theoretic way).	
	The standard uniform metric $d_\infty$ on $\C$ is defined by
    \begin{align}
    	d_\infty(A,B) := \max\limits_{x,y \in [0,1]}\lvert A(x,y)-B(x,y)\rvert.
    \end{align}
    It is well known that the metric space $(\C,d_\infty)$ is compact and that pointwise
    and uniform convergence of a sequence of copulas $(C_n)_{n \in \N}$ are equivalent.
    For more background on copulas and doubly stochastic measures we refer to the textbooks
	\cite{DuSe, NelsenCopulas}.	\\
	The Lebesgue decomposition (see \cite{Els,rudin}) of a doubly stochastic measure $\mu_C$ with respect 
    to $\lambda_2$ will be denoted by $$ \mu_C = \mu_C^{\ll} + \mu_C^{\perp},$$
    where $\mu_C^{\ll}$ denotes the absolutely continuous and $\mu_C^{\perp}$ the 
    singular component of $\mu_C$. 
    We will write $sing(C)=\mu_C^{\perp}(\unit^2) \in [0,1]$ for the total mass of the singular 
    component of $C$ and set $abs(C):=1-sing(C)$. 
   
	A mapping $K:\unit \times \borelf \to \unit$ will be called a Markov kernel if 
    $x \mapsto K(x,B)$ is measurable for every fixed $B\in \borelf$ and 
    $B \mapsto K(x,B)$ is a probability measure on $\borelf$ for every fixed $x \in \unit$.
    It is well known that for every copula $C \in \mathcal{C}$ there exists a Markov kernel 
    $K_C: \unit \times \borelf \to \unit$ fulfilling 
    \begin{align*}
    	\int_E K_C(x,F) d\lambda(x) = \mu_C(E \times F)
    \end{align*}
    for all $E,F \in \borelf$ and that this Markov kernel is unique for $\lambda$-almost 
    every $x \in [0,1]$. Vice versa, every Markov kernel $K:\unit \times \borelf \to \unit$  
    having $\lambda$ as invariant distribution, i.e., fulfilling 
    \begin{align*}
    	\int_{[0,1]} K(x,F) d\lambda(x) = \lambda(F)
    \end{align*}   
    for every $F \in \borelf$, can be shown to be the Markov kernel of a unique copula $C$, which then fulfills
    $$
    C(x,y)=\int_{[0,x]} K(t,[0,y]) d\lambda(t)
    $$    
    for all $x,y \in [0,1]$. Notice that for fixed $y \in [0,1]$, $\lambda$-almost every $x \in [0,1]$ is 
    a Lebesgue point (see \cite{rudin}) of the mapping $x \mapsto K(x,[0,y])$, so the identity 
    \begin{equation}
     \frac{\partial C(x,y)}{\partial x}=K(x,[0,y])
    \end{equation}     
    holds (for such $x$).         
    For more background on Markov kernels and disintegration we refer to \cite{Ka}, 
    for more results on the interplay 
    between Markov kernels and copulas, e.g., to \cite{KTFSTW, Trup6}.\\
    Following \cite{DuranteSLC} $C \in \CSL$ is called lower semilinear (LSL) copula, if for every 
	$x \in (0,1]$ the mappings $t \mapsto h_x(t) := C(t,x)$ and $ t \mapsto v_x(t) := C(x,t)$ are li\-near on $[0,x]$. 
    As already mentioned in the introduction, $\mathcal{D}$ denotes the family of all copula diagonals, i.e.,     
    we have that $\delta:[0,1]\to[0,1]$ is an element of $\mathcal{D}$ if, and only if
    it fulfills the following three conditions: 
  	\begin{itemize}
  		\item[(i)] $\delta(u) \leq u$ for all $u \in [0,1]$ and $\delta(1)=1$,
  		\item[(ii)] $\delta$ is non-decreasing,
  		\item[(iii)] $\delta$ is $2$-Lipschitz, i.e., 
  		$|\delta(v)-\delta(u)| \leq 2|v-u|$ holds for all $u,v \in [0,1]$.
  	\end{itemize}
  	As already mentioned in the introduction, according to \cite{DuranteSLC} LSL copulas correspond to special diagonals. 
  	In fact, letting $\DSL$ be defined according to equation (\ref{DLSL}), the following result holds
  	for an arbitrary diagonal $\delta \in \mathcal{D}$ and $S_\delta$ defined according to 
  	equation (\ref{semcopula1}) (with the convention $\frac{0}{0} := 0$):
  	 $S_\delta$ is a copula if, and only if $\delta \in \DSL$. \\
  	Again following \cite{DuranteSLC} and using the fact that Lipschitz continuous functions are differentiable 
  	$\lambda$-almost everywhere, it is straightforward
 	to check that for fixed $\delta \in \mathcal{D}$ we have $\delta \in \DSL$ if, and only if the following inequality 
 	holds $\lambda$-almost everywhere:
 	\begin{align}\label{abschaetzen}
 		\delta(x) \leq x \delta'(x) \leq 2\delta(x)
	\end{align}
	This immediately yields that every $\delta \in \DSL$ fulfills $$x^2= \delta_\Pi(x) \leq \delta(x) \leq \delta_M(x) = x$$
	for every $x \in [0,1]$, implying that $\Pi \leq C_\delta \leq M$ holds for every $S_\delta \in \CSL$.  
	
	In what follows, the so-called star product (a.k.a. Markov product) $A*B$ of copulas $A,B \in \mathcal{C}$
	will play a prominent role.   
	Letting $\partial_i$ denote the partial derivative with respect to the $i$-th coordinate, 	$A \ast B $ is defined by
 	\begin{align}\label{starprod}
 		(A\ast B)(x,y) := \int\limits_{[0,1]} \partial_2 A(x,s)\cdot 
 						\partial_1 B(s,y) \hspace*{1mm} d\lambda(s) 
 	\end{align}
 	for all $x,y \in [0,1]$. 
 	It is well known (see \cite{NelsenCopulas}) that $A \ast B $ is a copula, that the star product 
 	is in general not commutative, i.e., $A*B \neq B*A$ can hold, that $\Pi$ is the null- and $M$ the unit element in 
 	$(\mathcal{C},*)$, i.e., $A*\Pi=\Pi*A=\Pi$ and $A*M=M*A=A$ for every $A \in \mathcal{C}$.
 	A copula $A$ is called idempotent, if $A*A=A$ holds. 
 	Using Markov kernels it is straightforward to verify that the following identity holds:
 	\begin{align}\label{starprod2}
 		(A\ast B)(x,y) := \int\limits_{[0,1]} K_{A^t}(s,[0,x]) K_{B}(s,[0,y]) 
 			\hspace*{1mm} d\lambda(s) 
 	\end{align}
 	Moreover, according to \cite{TruFS}, using disintegration it can be shown that a (version of the) 
 	Markov kernel of $A*B$
 	is given by the standard composition of the Markov kernels of $A$ and $B$, a concept well-known 
 	in the context of Markov chains in discrete time. In other words, $K_A \circ K_B$, defined by
 	\begin{align}\label{starprod3}
 		(K_A \circ K_B)(x,F) = \int\limits_{[0,1]} K_{B}(s,F]) K_{A}(x,dy) 
 			\hspace*{1mm} d\lambda(s)
 	\end{align}
 	for every $x \in [0,1]$ and $F \in \mathcal{B}[0,1]$ is a (version of the) Markov kernel of $A*B$. 
 	For more background on the star product we refer to \cite{DarsowCMP,NelsenCopulas,TruFS,trutp10} and the references therein.

  \section{Markov kernels of LSL copulas and two important examples}\label{MarkovSection}
  	\noindent We start with recalling the form of the Markov kernel of LSL copulas. Letting 
  	$\delta \in \DSL$ be arbitrary but fixed, then there exists some set $\Lambda \in \mathcal{B}(\unit)$ with
 	$\lambda(\Lambda) = 1$ and some Borel measurable function $\hat{w}_\delta:[0,1] \rightarrow [0,2]$ 
 	such that for every $x\in \Lambda$ we have $\delta'(x) = \hat{w}_\delta(x)$. 
    Defining $w_\delta :\unit \to [0,2]$ (again using the convention $\frac{0}{0}:=0$) by
 	$$ w_\delta(x) = \hat{w}{_\delta}(x) \mathbf{1}_{\Lambda}(x) +  
 	\frac{\delta(x)}{x}\mathbf{1}_{\Lambda^c}(x)$$ 	
 	we have that $w_\delta$ is measurable, that $w_\delta=\delta'$
        holds $\lambda$-almost everywhere, and using 
 	inequality (\ref{abschaetzen}) that $w_\delta(x) \geq \frac{\delta(x)}{x}$
        for every $x \in [0,1]$.
 	We will refer to $\hat{w}_\delta$ and $w_\delta$ as measurable versions of the derivative of $\delta$. 
 	The following result has already been derived in \cite{SMPSconf}, we just recall it for the sake of completeness 
 	and since it will be used in the sequel.
       	
   \begin{Theorem}[\cite{SMPSconf}]\label{kern}
 	Let $\Sd$ be a LSL copula for a given $\delta \in \DSL$ and let $w_\delta$ the measurable version of 
 	the derivative of $\delta$ as constructed above.
 	Then (a version of) the Markov Kernel $K_{\Sd}$ of $\Sd$ is given by
 	\begin{equation}\label{eqKern}
 		K_{\Sd}(x,[0,y]) = \begin{cases}
 						\frac{y}{x} \, w_\delta(x) - \frac{y}{x^2}\delta(x) & 
 							\text{ if } y < x,\\
 						\frac{1}{y} \,\delta(y) &  \text{ if } y \geq x.
 						\end{cases}
 	\end{equation}
 	\end{Theorem}
    Notice that for fixed $x \in (0,1)$ the conditional distribution function 
    $y \mapsto F^\delta_x(y):=K_{S_\delta}(x,[0,y])$ is Lipschitz continuous on $[0,x)$ and on $[x,1]$. 
    As a direct consequence, the only possible discontinuity point of $F^\delta_x$ is $y=x$, i.e., 
    the only point mass the measure $K_{S_\delta}(x,\cdot)$ can have is at $y=x$. Using this simple observation 
    yields the following result going back to \cite{SMPSconf}.
 	\begin{Lemma}[\cite{SMPSconf}]\label{singmass} For every LSL copula $\Sd$ the singular mass is given by
 	\begin{align}
 	sing(S_\delta)=2\int\limits_{[0,1]} \frac{\delta(x)}{x} \hspace*{1mm}
 			d\lambda(x)-1
 	\end{align}
 	\end{Lemma}
 	 	
 	In the following example we introduce two important types of diagonals in $\DSL$, determine explicit 
 	expressions for the corresponding LSL copula and its Markov kernel. These diagonals will play a prominent role in 
 	Section 5 on concordance of LSL copulas.
 	\begin{example}\label{KerneAB}
 		For every $a \in \unit$ define the diagonals $l_a, u_a:\unit \to \unit$ by 
		\begin{align*}
			l_a(x) := 
			\begin{cases} 
				ax & \text{ if } x \leq a\\
				x^2 & \text{ if } x > a\\
			\end{cases} \qquad
			u_a(x) := 
			\begin{cases} 
				\frac{x^2}{a} & \text{ if } x \leq a\\
				x & \text{ if } x > a.\\
			\end{cases} \quad
		\end{align*}
		Figure \ref{fig:diagonals} depicts both diagonals for the case $a=\frac{1}{2}$.
	\begin{figure}[ht!]
 		\centering
 		\includegraphics[width=\textwidth]{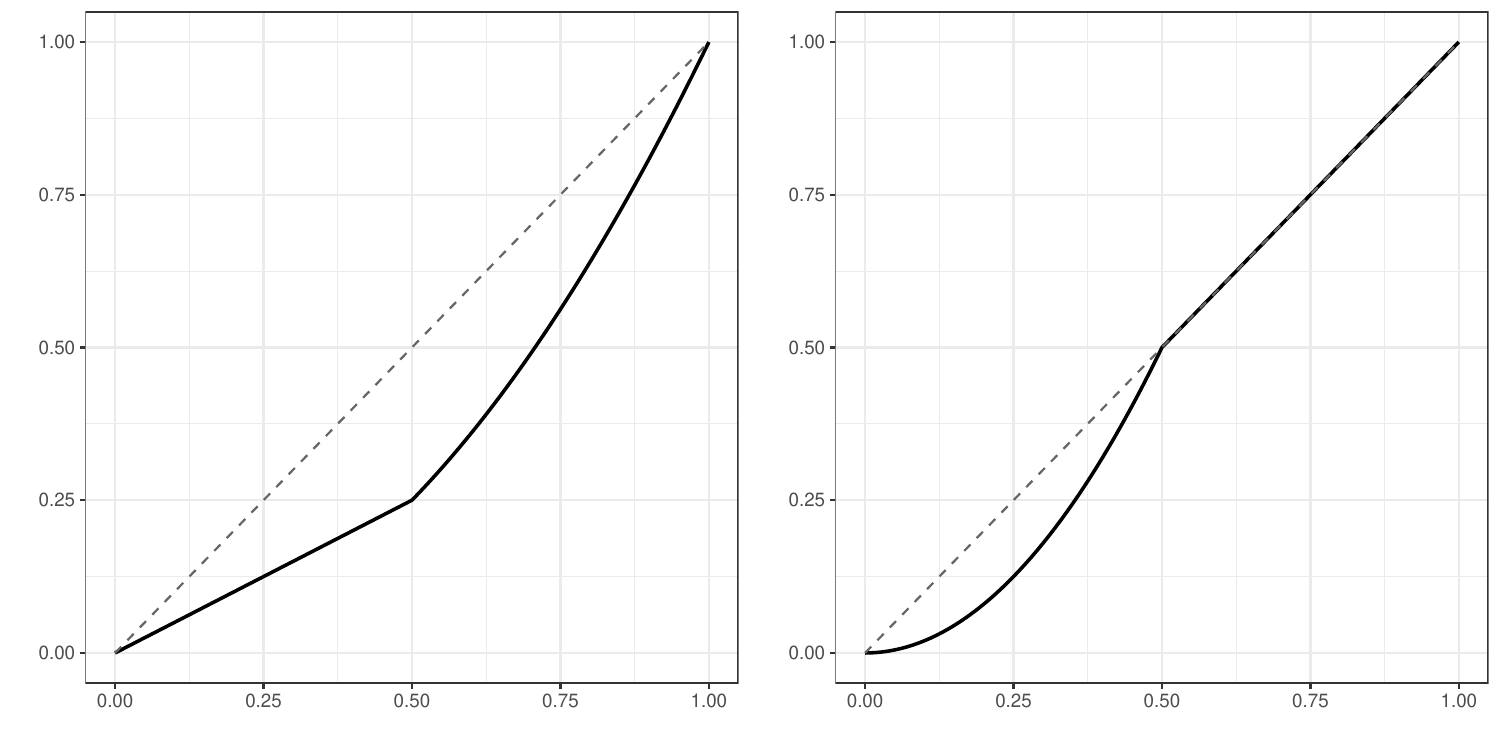}
		\caption{The diagonals $l_a,u_a$ for the case $a=\frac{1}{2}$ as considered in Example \ref{KerneAB}.}
 		\label{fig:diagonals}
 	\end{figure}
		It is straightforward to verify that $l_a,u_a \in \DSL$ for every $a \in [0,1]$ and that	
		the LSL copulas $S_{l_a},S_{u_a}$ induced by 	$l_a, u_a$ are given by
		\begin{align*}
		    S_{l_a}(x,y) &:= 
			\begin{cases} 
				\min\lbrace x,y \rbrace\cdot a & \text{ if } \max\lbrace x,y \rbrace 
													\leq a\\
				xy & \text{ if } \max\lbrace x,y \rbrace > a\\
			\end{cases} \\
			S_{u_a}(x,y) &:= 
			\begin{cases} 
				\frac{xy}{a} & \text{ if } \max\lbrace x,y \rbrace \leq a\\
				\min\lbrace x,y \rbrace & \text{ if }  \max\lbrace x,y \rbrace > a.\\
			\end{cases}
		\end{align*}
		Calculating the derivatives of $l_a, u_a$ and applying Theorem 
		\ref{kern} yields 
		\begin{align*}
		    K_{S_{l_a}}(x,[0,y]) &:= 
			\begin{cases} 
				0 & \text{ if } y < x \leq a \\
				a & \text{ if } x \leq y \leq a\\
				y & \text{ if } \max\lbrace x,y \rbrace > a\\
			\end{cases}\\
			K_{S_{u_a}}(x,[0,y]) &:= 
			\begin{cases} 
				0 & \text{ if } y < x, a < x \\
				\frac{y}{a} & \text{ if } \max\lbrace x,y \rbrace \leq a\\
				1 & \text{ if }  y \geq x, \,y > a.\\
			\end{cases}
		\end{align*}
     For the singular masses we obtain $sing(S_{l_a})=a^2$ as well as $sing(S_{u_a})=1-a$. 		
	The Markov kernels $K_{S_{l_a}}, K_{S_{u_a}}$ for $a= \frac{1}{2}$ are depicted in 
	 Figure \ref{fig:kernels}. 
	\begin{figure}[ht!]
 		\centering
 		\includegraphics[width=\textwidth]{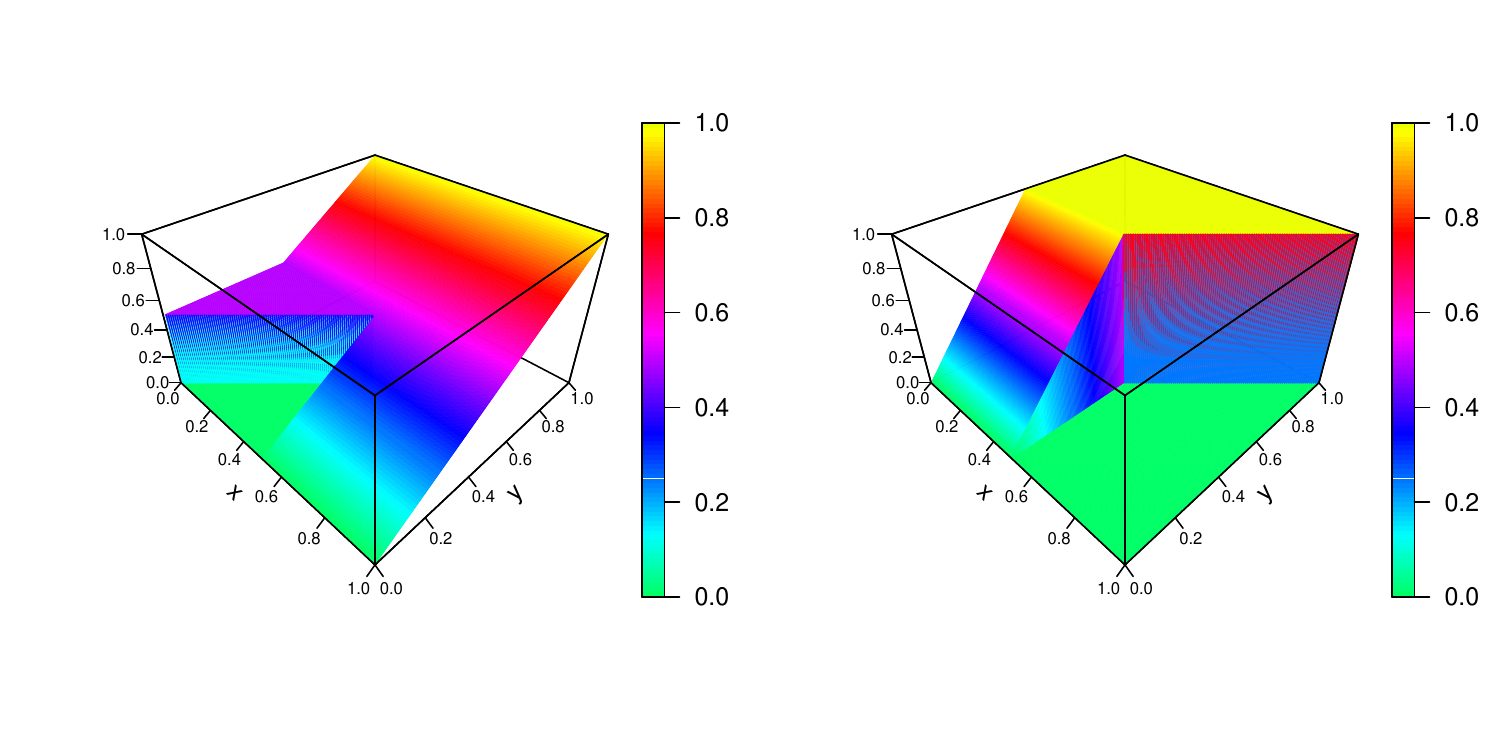}
		\caption{Surface plots of the functions $(x,y)\mapsto K_{S_{l_{a}}}(x,[0,y])$
 		and $(x,y)\mapsto K_{S_{u_{a}}}(x,[0,y])$ as considered in Example \ref{KerneAB} for the case $a=\frac{1}{2}$.}
 		\label{fig:kernels}
 	\end{figure} 
 	\end{example}
 	
	We complete this short section with some monotonicity properties of LSL copulas.
 	Recall (see \cite{NelsenCopulas} and \cite{Tp2}) that a copula $C \in \mathcal{C}$ is said to be
 	\begin{itemize}
 		\item positively quadrant dependent (PQD) if $C(x,y) \geq \Pi(x,y)$ holds for 
 				all $(x,y)\in [0,1]^2$.
 		\item left tail decreasing (LTD) if, for any $y\in [0,1]$, the mapping 
 				$(0,1)\to \mathbb{R}$ given by $x \mapsto \frac{C(x,y)}{x}$ is 
 				non-increasing.
 		\item stochastically increasing (SI) if, for (a version of) the Markov 
 				kernel $K_C$ and any $y\in (0,1)$ the mapping 
 				$x \mapsto K_C(x,[0,y])$ is non-increasing.
 	\end{itemize}
	We already know that every $C \in \CSL$ fulfills $\Pi \leq S_\delta \leq M$, hence LSL copulas are PQD
	(also see \cite{DuranteSLC}). \\
	Concerning LTD suppose that $\delta \in \DSL$ and $y \in (0,1)$ are fixed.
 	Then using inequality (\ref{abschaetzen}) obviously the mapping $$x \mapsto 
 	\frac{S_\delta(x,y)}{x} = \begin{cases} 
	y \,\frac{\delta(x)}{x^2} & y \leq x\\
	\frac{\delta(y)}{y} & y \geq x 	
 	\end{cases}$$ is non-increasing, so LSL copulas are LTD.\\
 	Finally, the following simple example shows that LSL copulas are not necessarily SI.
 	\begin{figure}[ht!]
 		\centering
 		\includegraphics[scale=0.46]{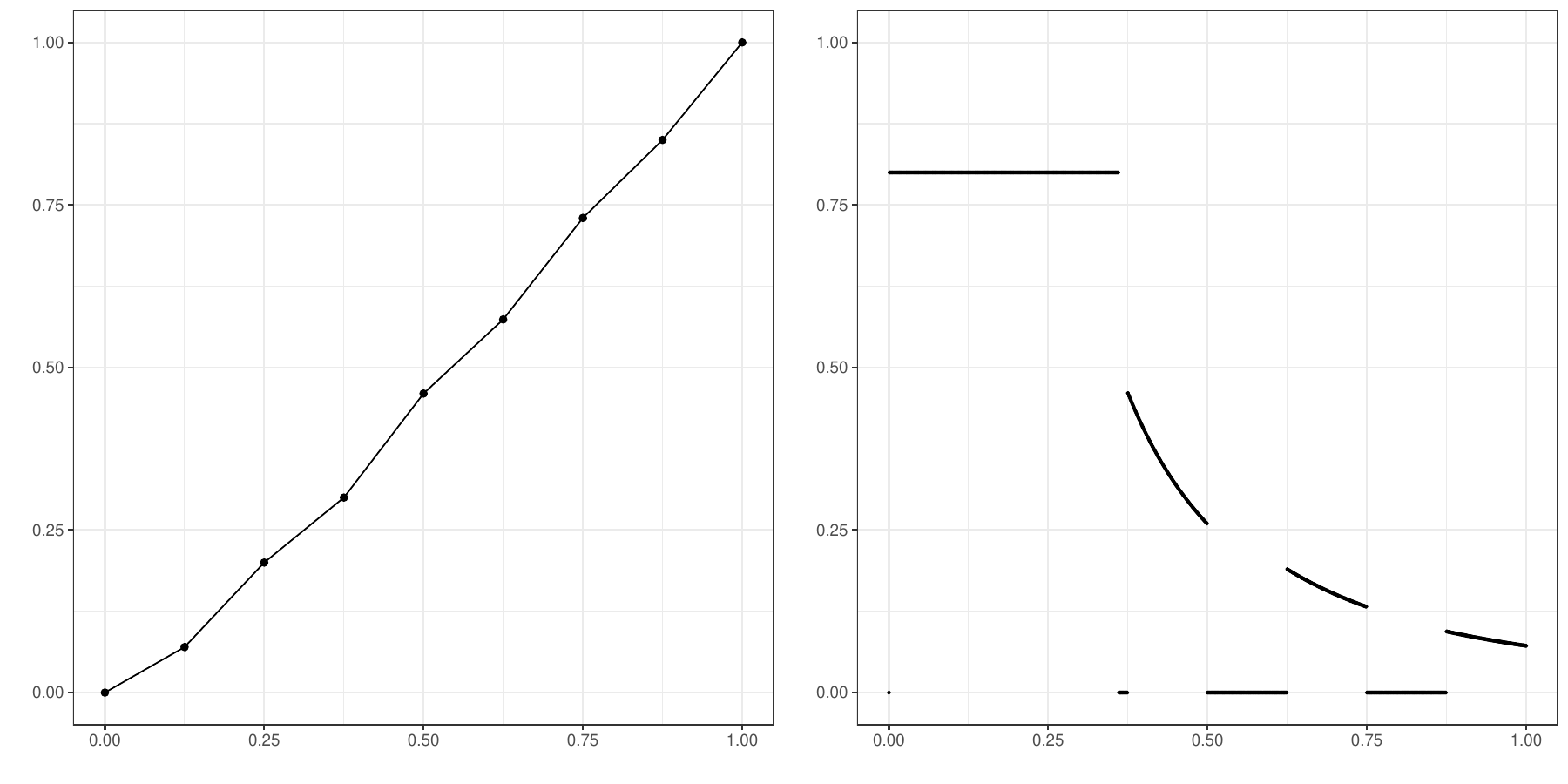}
 		\caption{The diagonal $\delta \in \DSL$ considered in Example \ref{ExHoles}
 		(left panel) and the mapping $x \mapsto K_{S_\delta}(x,[0,y])$ for $y=0.36$
 		(right panel).}
 		\label{NotSI} \vspace{0.1cm} 
 	\end{figure}

 	\begin{figure}[ht!]
 		\centering
 		\includegraphics[scale=0.63]{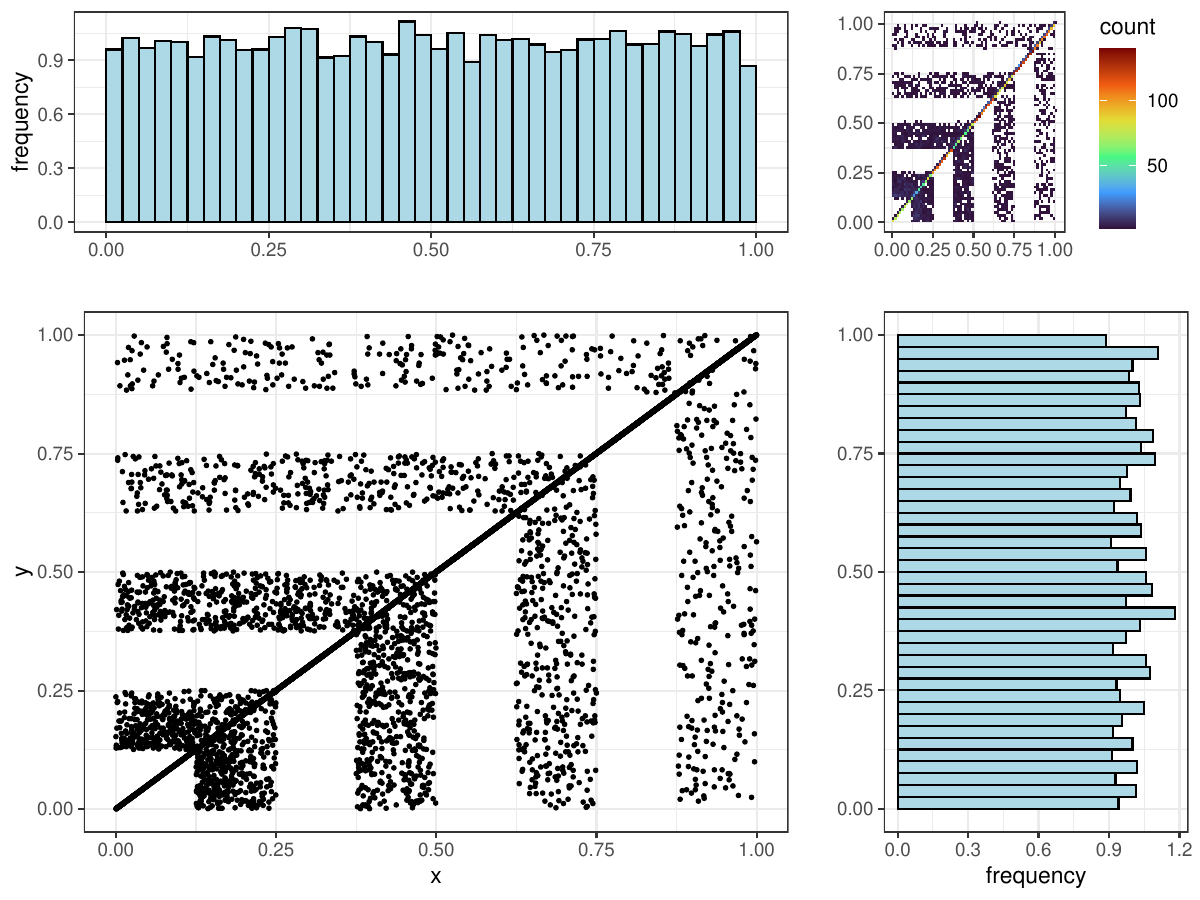}
 		\caption{Sample of size $n=10.000$ of the LSL copula $\Sd$ considered 
 		in Example \ref{ExHoles} (lower left panel); two-dimensional 
 		histogram (upper right panel) and marginal histograms 
 		(upper left and lower right panel).}
 		\label{sampleSI}
 	\end{figure}
 	
 	\begin{example}\label{ExHoles}
 		Consider the points $(0,0), (\frac{1}{8},\frac{7}{10}), (\frac{1}{4},\frac{1}{5}),
 		(\frac{3}{8},\frac{3}{10}), (\frac{1}{2},\frac{23}{50}), 
 		(\frac{5}{8},\frac{287}{500})$, $(\frac{3}{4},\frac{73}{100}),
 		(\frac{7}{8},\frac{17}{20}), (1,1)$ and let
		$\delta: \unit \to \unit$ denote the linear interpolation of these points.
		It is straightforward to verify that $\delta \in \DSL$ and that the mapping $x \mapsto 
		\frac{\delta(x)}{x}$ is constant on every second of the intervals formed by the $x$-coordinates
		of the afore-mentioned points. Figure  \ref{NotSI} depicts the diagonal $\delta$ and the mapping  
		$x \mapsto K_{S_\delta}(x,[0,y])$ for $y=0.36$, Figure \ref{sampleSI} a sample of the corresponding LSL copula $\Sd$.
		Obviously we can find $y\in (0,1)$ such that mapping $x \mapsto K_{\Sd}(x,[0,y])$ is not
		non-increasing. 
 	\end{example}
 	

	\section{The star product of LSL copulas}\label{starproduct}
	Contrary to Archimedean and Extreme Value copulas (two families also characterized in terms of 
	univariate functions) the familie $\CSL$ of LSL co\-pulas is closed with respect to the star product - the following 
	theorem holds (to keep notation simple we avoid again working with versions of the derivatives since 
	the integral ignores sets of $\lambda$-measure $0$):
 	\begin{Theorem}\label{sternprodukt}
	Suppose that $\delta_1,\delta_2 \in \DSL$. Then the star product $S_{\delta_1}\ast S_{\delta_2}$ is given by  
 	\begin{align}
 		(S_{\delta_1}\ast S_{\delta_2})(x,y) = \left\lbrace
 		\begin{array}{c c}
 		\displaystyle\frac{x}{y^2}\delta_1(y)\delta_2(y) + xy\displaystyle\int\limits_{[y,1]} 
 		 \left(\tfrac{\delta_1(u)}{ u}\right)'
 		\left(\tfrac{\delta_2(u)}{ u}\right)' \hspace*{1mm} d\lambda(u) & \text{if } 
 		y > x,\\
 		\displaystyle\frac{y}{x^2}\delta_1(x)\delta_2(x) + xy \displaystyle\int\limits_{[x,1]} 
 		 \left(\tfrac{\delta_1(u)}{ u}\right)'
 		\left(\tfrac{\delta_2(u)}{ u}\right)' \hspace*{1mm} d\lambda(u) & \text{if } 
 		y \leq x.\\
 		\end{array}\right.
 	\end{align}
 	In particular, $S_{\delta_1}\ast S_{\delta_2}$ is a LSL copula too, i.e. $S_{\delta_1}\ast S_{\delta_2} \in \CSL$.
 	\end{Theorem}
 	\begin{proof}
 	Suppose that $0 \leq x < y \leq 1$. Then using symmetry of LSL copulas and equation (\ref{starprod2}) we obtain
 	\begin{align*}
 		(S_{\delta_1}\ast S_{\delta_2})(x,y) &= \int\limits_{[0,x]} \tfrac{\delta_1(x)}{x}
 		\tfrac{\delta_2(y)}{y}\hspace*{1mm} d\lambda(t) 
 		+\int\limits_{[x,y]} \left(\tfrac{x}{t}w_{\delta_1}(t)-\tfrac{x}{t^2}\delta_1
 		(t)\right)\tfrac{\delta_2(y)}{y}\hspace*{1mm} d\lambda(t) \\
 		&\hspace*{0.8cm}+ \int\limits_{[y,1]} \left(\tfrac{x}{t}w_{\delta_1}(t)-
 		\tfrac{x}{t^2}\delta_1(t)\right)\left(\tfrac{y}{t}w_{\delta_2}(t)-
 		\tfrac{y}{t^2}\delta_2(t)\right)\hspace*{1mm} d\lambda(t) \\
 		&= \tfrac{1}{xy}\delta_1(x)\delta_2(y)x + \tfrac{x}{y}\delta_2(y)
 		\int\limits_{[x,y]} \left(\tfrac{\delta_1(t)}{t}\right)'\hspace*{1mm} d\lambda(t)\\
 		&\hspace*{0.8cm}+  xy\int\limits_{[y,1]} \left(\tfrac{\delta_1(t)}{t}\right)' 
 		\left(\tfrac{\delta_2(t)}{t}\right)'\hspace*{1mm} d\lambda(t)\\
 		&= \tfrac{1}{y}\delta_1(x)\delta_2(y) + \tfrac{x}{y}\delta_2(y)
 			\left(\tfrac{\delta_1(y)}{y}-
 		\tfrac{\delta_1(x)}{x}\right) + \\
 		&\hspace*{0.8cm}+ xy\int\limits_{[y,1]} \left(\tfrac{\delta_1(t)}{t}
 			\right)' 
 		\left(\tfrac{\delta_2(t)}{t}\right)'\hspace*{1mm} d\lambda(t)\\
 		&= \frac{x}{y^2}\delta_1(y)\delta_2(y) + xy\int\limits_{[y,1]} 
 		\left(\tfrac{\delta_1(t)}{t}\right)' 
 		\left(\tfrac{\delta_2(t)}{t}\right)'\hspace*{1mm} d\lambda(t).
 	\end{align*}
 	Analogously for $0 \leq y \leq x \leq 1$ we have
 	\begin{align*}
 		(S_{\delta_1}\ast S_{\delta_2})(x,y) &= \int\limits_{[0,y]} \tfrac{\delta_1(x)}{x}
 		\tfrac{\delta_2(y)}{y}\hspace*{1mm} d\lambda(t) 
 		+\int\limits_{[y,x]} \tfrac{\delta_1(x)}{x}\left(\tfrac{y}{t}w_{\delta_2}(t)-
 		\tfrac{y}{t^2}\delta_2(t)\right)\hspace*{1mm} d\lambda(t) \\
 		&\hspace*{0.8cm}+ \int\limits_{[x,1]} \left(\tfrac{x}{t}w_{\delta_1}(t)-
 		\tfrac{x}{t^2}\delta_1(t)\right)\left(\tfrac{y}{t}w_{\delta_2}(t)-
 		\tfrac{y}{t^2}\delta_2(t)\right)\hspace*{1mm} d\lambda(t) \\
 		&= \tfrac{1}{xy}\delta_1(x)\delta_2(y)y + \tfrac{y}{x}\delta_1(x)
 		\int\limits_{[y,x]} \left(\tfrac{\delta_2(t)}{t}\right)'\hspace*{1mm}d\lambda(t)\\
 		&\hspace*{0.8cm}+  xy\int\limits_{[x,1]} \left(\tfrac{\delta_1(t)}{t}\right)' 
 		\left(\tfrac{\delta_2(t)}{t}\right)'\hspace*{1mm} d\lambda(t)\\
 		&= \tfrac{1}{x}\delta_1(x)\delta_2(y) + \tfrac{y}{x}\delta_1(x)
 		\left(\tfrac{\delta_2(x)}{x}-
 		\tfrac{\delta_2(y)}{y}\right) \\
 		&\hspace*{0.8cm} + xy\int\limits_{[x,1]} 
 		\left(\tfrac{\delta_1(t)}{t}\right)' 
 		\left(\tfrac{\delta_2(t)}{t}\right)'\hspace*{1mm} d\lambda(t)\\
 		&= \frac{y}{x^2}\delta_1(x)\delta_2(x) + xy\int\limits_{[x,1]} 
 		\left(\tfrac{\delta_1(t)}{t}\right)' 
 		\left(\tfrac{\delta_2(t)}{t}\right)'\hspace*{1mm} d\lambda(t).
 	\end{align*}
 	Considering that for fixed $x \in (0,1]$ the two mappings
 	\begin{align*}
 		t \mapsto (S_{\delta_1}\ast S_{\delta_2})(t,x) &= t\bigg( \tfrac{1}{x^2}\delta_1(x)\delta_2(x) + 
 		x\int\limits_{[x,1]} 
 		 \left(\tfrac{\delta_1(u)}{ u}\right)'
 		\left(\tfrac{\delta_2(u)}{ u}\right)' \hspace*{1mm} d\lambda(u)\bigg),\\
 		t \mapsto (S_{\delta_1}\ast S_{\delta_2})(x,t) &= t\bigg( \tfrac{1}{x^2}\delta_1(x)\delta_2(x) + 
 		x\int\limits_{[x,1]} 
 		 \left(\tfrac{\delta_1(u)}{ u}\right)'
 		\left(\tfrac{\delta_2(u)}{ u}\right)' \hspace*{1mm} d\lambda(u)\bigg).
 	\end{align*}
 	are obviously linear $[0,x]$ it follows that	$S_{\delta_1}\ast S_{\delta_2} \in \CSL$ and the proof is complete.
 	\end{proof}
    Since the sets $\DSL$ and $\CSL$ are in one-to-one correspondence Theorem \ref{sternprodukt} implies that 
    the star product can be `translated' to the class $\DSL$. In fact, the diagonal of the copula 
    $S_{\delta_1}\ast S_{\delta_2}$ from Theorem \ref{sternprodukt} is obviously given by
    \begin{align*}
 			(S_{\delta_1}\ast S_{\delta_2})(x,x) = \frac{1}{x}\delta_1(x)\delta_2(x) 
 			+ x^2\int\limits_{[x,1]}
 			\left(\tfrac{\delta_1(u)}{ u}\right)'\left(\tfrac{\delta_2(u)}{ u}\right)' 
 			\hspace*{1mm} d\lambda(u).
 		\end{align*}
    This motivates the following definition.
    	\begin{Definition}\label{diagonalenprodukt}
 	For every pair $(\delta_1,\delta_2)$ of diagonals in $\DSL$ the star product $\delta_1 \ast \delta_2$
 	is defined by 
 	\begin{align}
 		(\delta_1 \ast \delta_2)(x) := \frac{1}{x}\delta_1(x)\delta_2(x) + 
 			x^2\int\limits_{[x,1]} 
 		\left(\tfrac{\delta_1(u)}{ u}\right)'
 		\left(\tfrac{\delta_2(u)}{ u}\right)' \hspace*{1mm} d\lambda(u),
 	 \end{align}
 	for every $x\in (0,1]$ as well as $(\delta_1 \ast \delta_2)(0):=0$. 
 	\end{Definition}
 	\begin{remark}
 	 Theorem \ref{sternprodukt} and the one-to-one correspondence of the families $\DSL$ and $\CSL$ mentioned above and in the 
 	 introduction imply that $\delta_1 \ast \delta_2 \in \DSL$ holds - since $S_{\delta_1}\ast S_{\delta_2}$ is 
 	 a LSL copula its diagonal is an element of $\DSL$. Moreover we obviously have the following identity for each 
 	 pair $(\delta_1,\delta_2) \in \DSL \times \DSL$, which we 
 	 will use various times in what follows:   
 	 \begin{equation}
 	   S_{\delta_1}\ast S_{\delta_2}=S_{\delta_1 \ast \delta_2}
 	 \end{equation}
 	 In other words, the mapping $\iota: \CSL \rightarrow \DSL$ assigning every LSL copula its diagonal
 	 (as well as its inverse $\iota^{-1}$) is an isomorphism with respect to the star product.  
 	\end{remark}
     Considering the diagonals $\delta_M$ and $\delta_\Pi$ of $M$ and $\Pi$, respectively, obviously the
 	following interrelations hold for every $\delta \in \DSL$:
 	\begin{align*}
 		\delta_\Pi \ast \delta &= \delta \ast \delta_\Pi = \delta_\Pi\\
 		\delta_M \ast \delta &= \delta \ast \delta_M = \delta
 	\end{align*} 	
 	
 	In what follows we will study the limit behavior of sequences $(S_\delta^{*n})_{n \in \mathbb{N}}$ where
 	$S_\delta^{*1}=S_\delta, S_\delta^{*2}=S_\delta * S_\delta, S_\delta^{*3}=S_\delta * S_\delta * S_\delta, \ldots$.
    Notice that for general copulas $A$, the sequence $(A^{*n})_{n \in \mathbb{N}}$ does not need to converge - 
    the simplest example being $W$ for which the sequence $(W^{*n})_{n \in \mathbb{N}}$ jumps between $W$ and $M$.  
    One can, however, show that the sequence 	$(A^{*n})_{n \in \mathbb{N}}$ is Cesáro convergent (even with respect to 
    a metric stronger than $d_\infty$), see \cite{trutp10} for more information.
    As we will show, for LSL copulas the sequence $(S_\delta^{*n})_{n \in \mathbb{N}}$ does converge - 	
 	the following simple but key lemma opens the door to deriving the just mentioned convergence without 
 	much technical ado. 
 	\begin{Lemma}\label{decreasingness}
 	 For all $\delta_1,\delta_2 \in \DSL$ the following inequality holds for every $t \in [0,1]:$
 	 \begin{align}
 		(\delta_1 \ast \delta_2) (t) \leq \min\lbrace \delta_1(t),\delta_2(t) \rbrace
 	 \end{align}
 	\end{Lemma}
 	\begin{proof}
 		To simplify notation write $I:= \int\limits_{[t,1]} \left(\frac{\delta_1(u)}{ u}\right)'
 				\left(\frac{\delta_2(u)}{ u}\right)' \hspace*{1mm} d\lambda(u)$.
 		Applying inequality (\ref{abschaetzen}) and the fact that $u \mapsto \frac{\delta_1(u)}{u^2}$ 
 		is non-increasing on $[t,1]$ yields
 		\begin{align*}
 			I &= \int\limits_{[t,1]} \left(\tfrac{\delta_1(u)}{ u}\right)'
 					\left(\tfrac{\delta_2(u)}{ u}\right)' \hspace*{1mm} d\lambda(u) 
 					=\int\limits_{[t,1]} \left(\tfrac{w_{\delta_1}(u)u -
 					\delta_1(u)}{u^2}\right)\left(\tfrac{\delta_2(u)}{ u}\right)' 
 					\hspace*{1mm} d\lambda(u)\\
 			&\leq \int\limits_{[t,1]} \left(\tfrac{\frac{2\delta_1(u)}{u}u -
 					\delta_1(u)}{u^2}\right)\left(\tfrac{\delta_2(u)}{ u}\right)'
 					\hspace*{1mm} d\lambda(u)
 					= \int\limits_{[t,1]} \left(\tfrac{\delta_1(u)}{u^2}\right)
 					\left(\tfrac{\delta_2(u)}{ u}\right)' \hspace*{1mm} d\lambda(u)\\
 			&\leq \int\limits_{[t,1]} \left(\tfrac{\delta_1(t)}{t^2}\right)
 					\left(\tfrac{\delta_2(u)}{ u}\right)' \hspace*{1mm} d\lambda(u)
 					= \tfrac{\delta_1(t)}{t^2} \int\limits_{[t,1]}
 					\left(\tfrac{\delta_2(u)}{ u}\right)' \hspace*{1mm} d\lambda(u)\\
 			&=  \frac{\delta_1(t)}{t^2}\left( 1 - 
 				\frac{\delta_2(t)}{t}\right).
 		\end{align*}
 		It therefore follows immediately that
 		\begin{align*}
 			(\delta_1 \ast \delta_2) (t) &=  \frac{1}{t}\delta_1(t)\delta_2(t) + 
 					t^2\cdot I	\leq \frac{1}{t}\delta_1(t)\delta_2(t) + 
 					t^2 \cdot \frac{\delta_1(t)}{t^2}\left( 1 - \frac{\delta_2(t)}{t}
 					\right) \\
 			&= \frac{1}{t}\delta_1(t)\delta_2(t) + \delta_1(t) - \frac{1}{t}
 				\delta_1(t)\delta_2(t)= \delta_1(t)
 		\end{align*}
 		holds for $t \in (0,1]$.
 		Since the case $(\delta_1 \ast \delta_2)(t)  \leq \delta_2(t)$ follows analogously we obtain the 
 		desired inequality
 		\begin{align*}
 			\delta_1\ast\delta_2(t)\leq \min\lbrace \delta_1(t),\delta_2(t) \rbrace.
 		\end{align*}
 	\end{proof}
 	Lemma \ref{decreasingness} has the following nice consequence: 
 	\begin{Theorem}\label{dinftyconv}
 		Suppose that $\delta \in \DSL$. Then there exists some 
            $\overline{\delta}\in \DSL$ such that the sequence
 		$(\delta^{\ast n})_{n \in \mathbb{N}}$ converges to $\overline{\delta}$ uniformly.
 		Moreover we have  
 		\begin{align}\label{conveq}
 		\lim\limits_{n\to \infty} d_\infty\left(S_{\delta}^{\ast n},S_{\overline{\delta}} \right)= 0
 		\end{align}
 		and the limit $S_{\overline{\delta}}$ is idempotent.
 	\end{Theorem}
 	\begin{proof}
        First of notice that Lemma \ref{decreasingness} implies that the sequence 
        $(\delta^{\ast n}(t))_{n \in \mathbb{N}}$ is monotonically non-increasing for every $t \in [0,1]$. 
        Considering $\delta_\Pi \leq \delta^{\ast n} \leq \delta_M$ the sequence is bounded so 
        it follows immediately that $(\delta^{\ast n}(t))_{n \in \mathbb{N}}$ converges to 
        some point $\overline{\delta}(t) \in [t^2,t]$. Since $t \in [0,1]$ was arbitrary and $\DSL$ is 
        (as closed subset of a family of Lipschitz continuous functions) compact with respect to the supremum norm 
        $\Vert \cdot \Vert_\infty$ on $[0,1]$ it follows that $\overline{\delta} \in \DSL$ and that 
        $$\lim_{n \rightarrow \infty} \Vert \delta^{\ast n} - \overline{\delta} \Vert_\infty=0.$$  
        The very form of LSL copulas according to equation (\ref{semcopula1}) implies that for every 
        $(x,y) \in [0,1]^2$ the sequence $(S_{\delta}^{\ast n}(x,y))_{n \in \mathbb{N}}$ converges to  
        $S_{\overline{\delta}}(x,y)$. Using Lipschitz continuity therefore yields equation (\ref{conveq})  
        and it remains to prove idempotence, which can easily be done as follows.
        	Convergence of the sequence $(S_{\delta^{\ast n}})_{n \in \mathbb{N}}$ to 
 		$S_{\overline{\delta}}$ with respect to $d_\infty$ implies Cesáro convergence, i.e., 
 		\begin{align*}
 			\lim\limits_{n\to \infty} d_\infty\left(\frac{1}{n}\sum\limits_{i=1}^n 
 				S_{\delta^{\ast i}}, S_{\overline{\delta}} \right) = 0
 		\end{align*}
 		holds. Applying \cite[Theorem 2]{trutp10} and using the fact that convergence with respect to 
 		the metric $D_1$ implies convergence with respect to $d_\infty$ yields that $S_{\overline{\delta}}$ is idempotent.
 	\end{proof}
 	\begin{example}
     We return to the diagonal $u_a$ from Example \ref{KerneAB}. Obviously the corresponding LSL copula $S_{u_a}$ 
 	is the ordinal sum of $\langle\Pi,M\rangle$ with respect to $\langle0,a,1\rangle$ (see \cite{DuSe, NelsenCopulas} for 
 	background on ordinal sums) and as such idempotent. 
 	\end{example}
 	It turns out that all idempotent LSL copulas are of the form $S_{u_a}$ for some $a \in [0,1]$, so 
 	the family of idempotent LSL copulas is quite small and fully determined by one parametric function. 
 	Notice that, on the contrast, general idempotent copulas can be very diverse and complex - in fact, there are
 	idempotent copulas with fractal support, see \cite{TruFS}.   
 	\begin{Theorem}
	The following two conditions are equivalent for $\delta \in \DSL$:
	\begin{enumerate}
		\item $\delta * \delta = \delta$.
		\item There exists some $a\in [0,1]$ such that $\delta=u_a$.
	\end{enumerate}
	\end{Theorem}
	\begin{proof}
	It has already been mentioned that every LSL copula $S_{u_a}$ is idempotent, it suffices to show that the first 
	assertion implies the second one. 
	Consider the set $$F_\delta:=\{x \in (0,1): \delta(x)=x\}$$ and distinguish two cases: \\
	(i) $F_\delta= \emptyset$: In this case for every $t \in (0,1)$ there exists 
	some $s \in [t,1]$ fulfilling $\left(\frac{\delta}{id} \right)'(s)>0$. In fact, 
	if this was not the case then we could find some $t \in [0,1)$ fulfilling
	that for $\lambda$-almost every $s \in [t,1]$ we have
	$\left(\frac{\delta}{id} \right)'(s)=0$. This, however, implies 
	$$
	0 = \int_{[t,1]} \left(\frac{\delta}{id} \right)' d\lambda = \frac{\delta(1)}{1}
	 - \frac{\delta(t)}{t} = 1- \frac{\delta(t)}{t},
	$$ 
	which obviously contradicts $F_\delta= \emptyset$. \\
	From the proof of the inequality $(\delta_1 * \delta_2)(t) \leq \min\{\delta_1(t),
	\delta_2(t)\}$ we know that in case of $(\delta * \delta)(t)=\delta(t)$ for 
	every $s\in[t,1)$ with $\left(\frac{\delta}{id} \right)'(s)>0$ we have 
	$$
	0 < \left(\frac{\delta}{id} \right)'(s) = \frac{s \delta'(s)-\delta(s)}{s^2} 
	= \frac{\delta(s)}{s^2} = \frac{\delta(t)}{t^2}. 
	$$
	As a direct consequence, for $0 < t_1 < t_2 < 1$ and every $s \in [t_2,1)$ with 
	$\left(\frac{\delta}{id} \right)'(s)>0$ it follows that
	$$
	\frac{\delta(t_2)}{t_2^2} = \frac{\delta(s)}{s^2} = \frac{\delta(t_1)}{t_1^2}.  
	$$ 
	Since such $s$ exist arbitrarily close to $1$ and we have $\frac{\delta(1)}{1}=1$ 
	it follows that the function $s \mapsto \frac{\delta(s)}{s^2}$ is identical to 
	$1$ on the whole interval $(0,1)$. This shows the $\delta=u_1$
	and completes the proof for $F_\delta= \emptyset$. \\
	(ii) If $F_\delta \neq \emptyset$ then according to \cite{DuranteSLC} for every 
	$s \in F_\delta$ we even have $[s,1] \subseteq F_\delta$. 
	Set $a_0:= \inf\{t \in (0,1): \delta(t)=t\}$. Since for $a_0=0$ it follows that 
	$F_\delta=[0,1]$, implying $\delta=u_0$, it suffices to consider $a_0 > 0$. 
	Proceeding analogously to the proof of (i) we conclude that the function 
	$s \mapsto \frac{\delta(s)}{s^2}$ is constant on the interval $[0,a_0]$, which 
	finally yields $\delta=u_{a_0}$.
	\end{proof}
    We have already mentioned before that $\CSL$ is closed with respect to the star product.  	
    We close this section with showing that the star product of two non-LSL copulas may be a 
    LSL copula and prove the somewhat surprising fact that the star product of every Marshall
 	Olkin copula $M_{\alpha,\beta}$ with its transpose is a LSL copula.
 	Recall that for $\alpha, \beta \in [0,1]$ the Marshall-Olkin copula $M_{\alpha,\beta}$ is given by 
 	$$ M_{\alpha,\beta} (u,v) = \min\lbrace u^{1-\alpha}v, uv^{1-\beta}\rbrace,$$
 	so $M_{\alpha,\beta}$ is in general not a LSL copula. 
 	According to \cite{FuchsMO} 
 	$M_{\beta,\alpha} \ast M_{\alpha,\beta}$ is given by
 	\begin{align}\label{markovMO}
 		M_{\beta,\alpha} \ast M_{\alpha,\beta} = \begin{cases} 
		\Pi + \frac{\alpha^2}{1-2\alpha}\Pi\left( 1- (\Pi)^{\frac{\beta-2\alpha\beta}
		{\alpha}} M^{\frac{2\alpha\beta-\beta}{\alpha}}\right) & 
		\alpha \notin \lbrace 0,\frac{1}{2},
			1\rbrace,\\
		\Pi & \alpha = 0,\\
		\Pi + \frac{\beta}{2}\Pi\left(\log(M)-\log(\Pi)\right) & \alpha = \frac{1}{2},\\
		M_{\beta,\beta} & \alpha = 1.
 		\end{cases}
 	\end{align}
 	It is straightforward to verify that for all $x\in (0,1]$ the mapping
 	\begin{align*}
 	t \mapsto (M_{\beta,\alpha} \ast M_{\alpha,\beta}) (t,x) = \begin{cases} 
 			t\left(x +\frac{\alpha^2}{1-2\alpha}x-\frac{\alpha^2}{1-2\alpha}x^{\frac{\beta-2\alpha\beta}{\alpha}+1}
 			\right) & \alpha \notin \lbrace 0,\frac{1}{2},1\rbrace,\\
 			tx & \alpha = 0,\\
 			t\left( x - \frac{\beta}{2}x\log(x)\right) & \alpha = \frac{1}{2},\\
 			tx^{1-\beta} & \alpha = 1,
 			\end{cases}
 	\end{align*}
 	is linear on $[0,x]$. Considering symmetry of $M_{\beta,\alpha} \ast M_{\alpha,\beta}$ it follows that 
 	$t \mapsto  M_{\beta,\alpha} \ast M_{\alpha,\beta} (x,t)$ is linear on $[0,x]$ as well. We have therefore shown the 
 	following result: 
 	\begin{Theorem}
 		For every Marshall-Olkin copula $M_{\alpha,\beta}$ the star product 
 		$M_{\beta,\alpha} \ast M_{\alpha,\beta}$ is a LSL copula. 
 	\end{Theorem}
 	

 \section{Concordance of LSL copulas}
 \noindent We conclude this paper by studying concordance of LSL copulas and 
 investigating the exact region $\Omega^{LSL}$ determined by Kendall's $\tau$ and 
 Spearman's $\rho$, which is given by
    \begin{equation}\label{omega}
    \Omega^{LSL}:=\left\{(\tau(S_\delta),\rho(S_\delta)): \, S_\delta \in \CSL \right\}.
    \end{equation}	
 	Recall that given a pair $(X,Y)$ of random variables with continuous joint distribution function $H$ both, 
    Kendall's $\tau$ and Spearman's $\rho$ only depend on the unique copula $C$ underlying $(X,Y)$ and the 
    following formulas hold (see \cite{NelsenCopulas, SPT}): 	
 	\begin{align}
 		\tau(C) &:= 4 \int\limits_{[0,1]^2} C(u,v) \hspace*{1mm} d\mu_C(u,v) -1 \nonumber \\
 		\rho(C) &:= 12 \int\limits_{[0,1]^2} C(u,v) \hspace*{1mm} d\lambda_2(u,v) -3.
 		\label{rhoalg}
 	\end{align}
 	\subsection{Kendall's $\tau$ and Spearman's $\rho$}
    	For LSL copulas the formulas (\ref{rhoalg}) boil down to integrals only 
        involving the corresponding diagonal, the 
    	following result holds:
 	\begin{Lemma}\label{sprho}
 		For every LSL copula $\Sd \in \CSL$ Spearman's $\rho$ is given by
 		\begin{align}
 			\rho(\Sd) = 12\int\limits_{[0,1]} \delta(x)x \hspace*{1mm} d\lambda(x) - 3.
		\end{align}
 	\end{Lemma}
 	\begin{proof} 
		Plugging in $S_\delta$ and using symmetry we obtain
 		\begin{align*}
 			\rho(\Sd) &= 12\int\limits_{[0,1]^2}\Sd(x,y)\hspace*{1mm}d\lambda_2(x,y)-3
 						= 12 \int\limits_{[0,1]}\int\limits_{[0,1]} \Sd(x,y) \hspace*{1mm}
 							d\lambda(y)d\lambda(x)-3\\
 						&= 12 \int\limits_{[0,1]} \bigg( \int\limits_{[0,x]} y 
 							\tfrac{\delta(x)}{x} \hspace*{1mm} d\lambda(y) 
 							+ \int\limits_{[x,1]} x \tfrac{\delta(y)}{y}
 							\hspace*{1mm} d\lambda(y)\bigg) d\lambda(x) - 3\\
 						&= 12 \int\limits_{[0,1]} \bigg( 2 \cdot \int\limits_{[0,x]} y 
 							\tfrac{\delta(x)}{x} \hspace*{1mm} d\lambda(y) \bigg) 
 							d\lambda(x) - 3 = 24 \int\limits_{[0,1]} \tfrac{\delta(x)}{x} 
 							\int\limits_{[0,x]} y 
 							\hspace*{1mm} d\lambda(y) d\lambda(x) - 3\\
 						&= 24 \int\limits_{[0,1]} \tfrac{\delta(x)}{x} \tfrac{x^2}{2} 	
 							d\lambda(x)-3
 						= 12\int\limits_{[0,1]} \delta(x)x \hspace*{1mm} d\lambda(x) - 3.
 		\end{align*}
 	\end{proof}
    In order to derive a simple expression for Kendall's $\tau$ of LSL copulas we will use the subsequent  	
 	handy identity which can be proved via disintegration (see, e.g., \cite{NelsenCopulas}): 
 	\begin{align}\label{mua}
 		\int\limits_{[0,1]^2} B \hspace*{1mm} d\mu_A = \frac{1}{2} - 
 			\int\limits_{[0,1]^2} K_B(x,[0,y]) K_{A^t}(y,[0,x]) 
 			\hspace*{1mm} d\lambda_2(x,y)
 	\end{align}
 	\begin{Lemma}\label{kendaltau}
 		For every LSL  copula $\Sd \in \CSL$ Kendall's $\tau$ 
 		is given by
 		\begin{align}
 			\tau(\Sd) = 4\int\limits_{[0,1]}\frac{\delta(x)^2}{x} 
 							\hspace*{1mm}d\lambda(x) -1.
		\end{align}
 	\end{Lemma}
 	\begin{proof} Consider an arbitrary $\Sd \in \CSL$ with diagonal $\delta \in \DSL$ and 
			let $w_\delta$ denote the measurable version of $\delta'$ as 
			constructed in Section \ref{MarkovSection}. 
			Symmetry of LSL-copulas implies that the Markov kernel $K_{\Sd^t}$ of 
			$\Sd^t$ coincides with the Markov kernel $K_{\Sd}$ of $\Sd$. 
			Considering equations (\ref{mua}) and using the symmetry of $\Sd$ 
			the desired identity follows via
 		\begin{align*}
 			\tau(\Sd)  &= 4\int\limits_{[0,1]^2} \Sd(x,y) \hspace*{1mm} d\mu_{\Sd}(x,y) - 1 \\
 					&= 4\bigg(\frac{1}{2} - \int\limits_{[0,1]^2} K_{\Sd}(x,[0,y]) 
 						K_{\Sd^t}(y,[0,x]) \hspace*{1mm} d\lambda_2(x,y)\bigg) - 1\\
 					&= 1 - 4\int\limits_{[0,1]} \int\limits_{[0,1]} K_{\Sd}(x,[0,y]) 
 						K_{\Sd^t}(y,[0,x]) \hspace*{1mm} d\lambda(y)d\lambda(x)\\
 					&= 1 - 4 \int\limits_{[0,1]} \bigg( \int\limits_{[0,x]} 
 						\left(\tfrac{y}{x}w_\delta(x) - \tfrac{y}{x^2}\delta(x)\right)
 						\tfrac{1}{x}\delta(x)\hspace*{1mm} d\lambda(y) \\
 					&\hspace*{1.8cm} + \int\limits_{[x,1]} \left( \tfrac{1}{y}\delta(y)
 						\left(\tfrac{x}{y}w_\delta(y)- \tfrac{x}{y^2} \delta(y) \right) 
 						\right) d\lambda(y)\bigg)\hspace*{1mm} d\lambda(x)\\
 					&= 1 - 4 \int\limits_{[0,1]} 2\cdot\int\limits_{[0,x]} 
 						\tfrac{y}{x^2}w_\delta(x)\delta(x) - \tfrac{y}{x^3}\delta(x)^2 
 						\hspace*{1mm} d\lambda(y) d\lambda(x) \\
 					&= 1 - 8\int\limits_{[0,1]} \tfrac{1}{x^2} w_\delta(x)\delta(x) - 
 						\tfrac{1}{x^3}\delta(x)^2 \int\limits_{[0,x]} y\hspace*{1mm} 
 						d\lambda(y) d\lambda(x) \\
 					&= 1 - 4\int\limits_{[0,1]} w_\delta(x)\delta(x) - 
 						\tfrac{1}{x}\delta(x)^2 \hspace*{1mm}d\lambda(x) \\
 					&= 1 - 4\int\limits_{[0,1]} w_\delta(x)
 						\delta(x)  \hspace*{1mm}d\lambda(x) + 4\int\limits_{[0,1]}
 						\tfrac{\delta(x)^2}{x} \hspace*{1mm}d\lambda(x) \\
 					&= 1 - 4\int\limits_{[0,1]} u \hspace*{1mm}du +  4\int\limits_{[0,1]} 
 						\tfrac{\delta(x)^2}{x} \hspace*{1mm}d\lambda(x)
 						= 4\int\limits_{[0,1]}\tfrac{\delta(x)^2}{x} 
 						\hspace*{1mm}d\lambda(x) -1.
 		\end{align*}
 	\end{proof}
 	\begin{remark}
   Since some other measures of association might also be of interest in the context of applications, Lemma 
   \ref{gammabeta} gathers the resulting formulas for Ginni's $\gamma$, 
   Spearman's footrule $\phi$ and Blomqvist's $\beta$  of LSL copulas. 
    
   \end{remark} 	
 	\begin{example}\label{ex:deltaab}
 		We again return to the diagonals $l_a,u_a \in \DSL$ considered in Example
 		\ref{KerneAB}. Applying Lemma \ref{sprho} and \ref{kendaltau} directly yields
		$$\tau(S_{l_a}) = \rho(S_{l_a}) = a^4$$
		$$\tau(S_{u_a}) = 1-a^2, \hspace*{2mm} \rho(S_{u_a}) = 1-a^3$$
 	\end{example}
	It is well-known (and also follows directly from equation (\ref{rhoalg})) that Spearman's $\rho$ preserves 
	convex combinations, i.e., 
	$$ \rho(\alpha A + (1-\alpha) B) = \alpha\rho(A) + (1-\alpha) \rho(B) $$
	holds for $\alpha \in \unit$ and $A,B \in \C$. \\
    For Kendall's $\tau$ the situation is different, in general it does neither preserve 
    convex combinations, not even  	
	$$ \tau(\alpha A + (1-\alpha)B) \leq \alpha \tau(A) + (1-\alpha)\tau(B) $$
	needs to hold. 
		In fact, a straightforward calculation (also see \cite{FuchsSchmidt}) shows that 
		for the Fréchet family  \\
		$$\mathcal{F} := \left\lbrace \alpha W + \beta M + (1-\alpha - \beta)\Pi: \, \alpha, \beta \in [0,1],
		\alpha + \beta \leq 1 \right\rbrace$$ we have  
		\begin{align*}
			\tau(\alpha W + (1-\alpha - \beta)\Pi + \beta M) = \tfrac{(\beta-\tau)(\beta +
			\alpha + 2)}{3},
		\end{align*}
		hence, considering $\beta = 0$, $\alpha = \frac{1}{4}$, we obtain 
		$$ \tau\left( \tfrac{1}{4}\, W + \tfrac{3}{4} \, \Pi \right) >  \tfrac{1}{4}\,\tau\left(
		W \right) + \tfrac{3}{4}\, \tau\left(\Pi \right).$$
	For LSL copulas, however, Kendall's $\tau$  interpreted as function mapping $\CSL$ to $[0,1]$
	is strictly convex as the following result shows: 
	\begin{Lemma}\label{strictconvex}
		For $\delta_1,\delta_2 \in \DSL$ with $\delta_1 \neq \delta_2$ and 
		$\alpha \in [0,1]$ the following inequality holds:
		\begin{align}
			\tau(\lambda S_{\delta_1} + (1-\lambda)S_{\delta_2}) < 
			\lambda \tau(S_{\delta_1}) + (1-\lambda)\tau(S_{\delta_2})
		\end{align}		
	\end{Lemma}
	\begin{proof} 
		If suffices to prove the result for $\alpha=\frac{1}{2}$. 
		Suppose that $\delta_1,\delta_2 \in \DSL$ fulfill $\delta_1 \neq \delta_2$ and
		set $\tilde{\delta} = \frac{1}{2}\delta_1 + \frac{1}{2}\delta_2$. Strict
		convexity of the mapping $x \mapsto x^2$ in combination with Lipschitz continuity of 
		diagonals and the assumption $\delta_1 \neq \delta_2$ yields
		\begin{align*}
			\tau\left( S_{\tilde{\delta}}\right) &= 4 \int\limits_{[0,1]} 		
				\tfrac{1}{x} \tilde{\delta}(x)^2 \hspace*{1mm} d\lambda(x)  -1
			= 4 \int\limits_{[0,1]} \tfrac{1}{x} \left( \tfrac{1}{2}\delta_1(x) + 
				\tfrac{1}{2}\delta_2(x)\right)^2 \hspace*{1mm} d\lambda(x) -1 \\
			&< 4 \int\limits_{[0,1]} \tfrac{1}{x}\left(\tfrac{1}{2}\delta_1^2(x) + 
				\tfrac{1}{2}\delta_2^2(x)\right) \hspace*{1mm} d\lambda(x)-1 \\
			&= 4 \int\limits_{[0,1]} \tfrac{1}{2}\tfrac{\delta_1^2(x)}{x} \hspace*{1mm} 
			d\lambda(x) -\tfrac{1}{2} + 4 \int\limits_{[0,1]}
				\tfrac{1}{2}\tfrac{\delta_2^2(x)}{x} \hspace*{1mm} d\lambda(x)-\tfrac{1}{2}\\
			&= \tfrac{1}{2}\bigg(4 \int\limits_{[0,1]} \tfrac{\delta_1^2(x)}{x} \hspace*{1mm} 
			d\lambda(x) - 1 \bigg) + \tfrac{1}{2}
            \bigg(4 \int\limits_{[0,1]} 
			\tfrac{\delta_2^2(x)}{x} \hspace*{1mm} d\lambda(x) - 1 \bigg)\\
			&= \tfrac{1}{2}\tau\left( S_{\delta_1}\right)+
			\tfrac{1}{2}\tau\left( S_{\delta_2}\right).
		\end{align*}
	\end{proof}

 	\subsection{The $\tau$-$\rho$-region $\Omega^{LSL}$ determined by $\CSL$}
 	\noindent 
 	Building upon the derived formulas for $\tau$ and $\rho$ according to Lemma \ref{sprho} and
 	Lemma \ref{kendaltau} we now study the exact region $\Omega^{LSL}$ 
 	determined by Kendall's $\tau$ and Spearman's $\rho$ of LSL copulas and defined by
    \begin{equation}\label{def:taurho}
      \Omega^{LSL}:=\left\{(\tau(S_\delta),\rho(S_\delta)): \, S_\delta \in \CSL \right\}.
    \end{equation}     	
 	Possibly triggered by the paper \cite{SPT} in which the exact $\tau$-$\rho$-region for the full class $\mathcal{C}$
 	was derived, several papers on the regions determined by pairs of dependence measures (considering the 
 	full class $\mathcal{C}$ or specific important subclasses) appeared in the past ten years. 
 	For more details we refer, e.g., to the papers 
 	\cite{kokol2023exact, bukovvsek2024exact, bukovvsek2022exact, bukovvsek2021spearman, Mroz} and the references therein.
 	
 	Considering that $\Pi \leq \Sd \leq M$ holds for 	every $\Sd \in \CSL$ we obviously have 
 	$\Omega^{LSL} \subseteq \unit^2$. 
 	Due to countless simulations we conjecture that $\Omega^{LSL}$ is given by 
     \begin{equation}\label{ew:lower.upper1}
       R=\left\{(x,y) \in [0,1]^2: x \leq y \leq 1 - (1-x)^{\frac{3}{2}}  \right\}.
    \end{equation} 	
    \begin{figure}[ht!]
		\centering
		\includegraphics[scale=0.6]{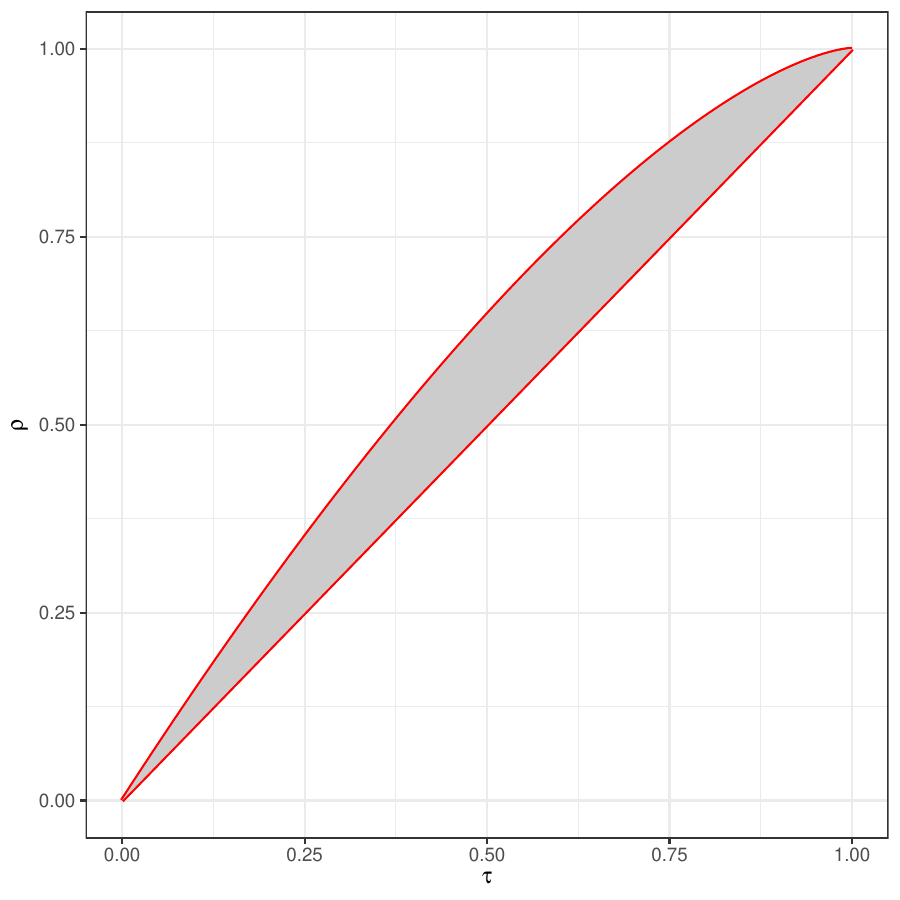}
		\caption{The conjectured $\tau$-$\rho$-region $\Omega^{LSL}$.}
		\label{fig:region1}
	\end{figure}
    In other words, writing
	\begin{align*}
		\Phi_l:\unit \to \unit, \hspace*{0.5cm}\Phi_l(x)&=x,\\
		\Phi_u:\unit \to \unit, \hspace*{0.5cm}\Phi_u(x)&=1-(1-x)^{\frac{3}{2}}
	\end{align*}	
	we conjecture that 
	$$
	\rho(S_\delta) \in [\Phi_l(\tau(S_\delta)),\Phi_u(\tau(S_\delta))] = 
	\left[\tau(S_\delta),\,1-(1-\tau(S_\delta))^{\frac{3}{2}}\right]
	$$
	holds for every $S_\delta \in \CSL$. The conjectured set $R$ is depicted in Figure \ref{fig:region1}.
    We have only been able to prove the lower inequality and to show that it is sharp (i.e., best possible), 
    the upper one remains an open question.  
    We now provide a proof for the lower bound, verify its sharpness, and then, despite not knowing 
    the upper bound, show that $\Omega^{LSL}$ is convex and compact.   	
 	\begin{Theorem}
	For every $\Sd \in \CSL$ the inequality
		\begin{align}\label{eq1}
			\tau(\Sd) &\leq \rho(\Sd)
		\end{align}
	holds. Moreover, inequality (\ref{eq1}) is sharp, i.e., for every $x \in \unit$ there exists some 
	LSL copula $S_\delta$ fulfilling $\tau(S_{\delta}) = \rho(S_\delta) = x$.
	\end{Theorem}
   \begin{proof}
	Let $\Sd \in \CSL$ be arbitrary but fixed. According to \cite{FredNelsen} the following identity 
	for $\rho(C) - \tau(C)$ holds for every copula $C \in \mathcal{C}$
	(notice that the integrand may only be defined on a set $E \in \mathcal{B}([0,1]^2)$ fulfilling $\lambda_2(E)=1$):
	\begin{align}\label{Nelsencharac}
		\tfrac{1}{4}(\rho(C) - \tau(C)) = \int\limits_{[0,1]^2} C(x,y) - 
			x\tfrac{\partial C(x,y)}{\partial x} 
			- y\tfrac{\partial C(x,y)}{\partial y} + \tfrac{\partial C(x,y)}{\partial x}
				\tfrac{\partial C(x,y)}{\partial y} \hspace*{1mm} d\lambda_2(x,y).
	\end{align}
	Using the fact that LSL copulas are symmetric, using Markov kernels equation (\ref{Nelsencharac}) boils down to
	\begin{equation*}\label{NelsencharacMarkov}
	\begin{split}
		\tfrac{1}{4}(\rho(\Sd) - \tau(\Sd)) &= \int\limits_{[0,1]^2} \Sd(x,y)+ 
				K_{\Sd}(x,[0,y])( K_{\Sd}(y,[0,x])- x)\\
				&\hspace*{2cm} - yK_{\Sd}(y,[0,x]) \hspace*{1mm} d\lambda_2(x,y).
	\end{split}
	\end{equation*}
	For proving $\tau(\Sd) \leq \rho(\Sd)$ it therefore suffices to show that the last integrand is non-negative, which 
	can be done as follows: For $y < x$ it follows that
	\begin{align*}
		\Sd(x,y)+ &K_{\Sd}(x,[0,y])( K_{\Sd}(y,[0,x])- x) - yK_{\Sd}(y,[0,x]) \\
		&=
		y\frac{\delta(x)}{x} + K_{\Sd}(x,[0,y])\left(\tfrac{\delta(x)}{x}-x\right)
			-y\tfrac{\delta(x)}{x}\\
		&= K_{\Sd}(x,[0,y])\left(\tfrac{\delta(x)-x^2}{x}\right)\\
		&\geq 0.
	\end{align*}
	The case $y > x$ follows directly from the symmetry of LSL copulas, so the proof of inequality  (\ref{eq1}) is complete.
	The assertion on sharpness is a direct consequence of Example \ref{ex:deltaab}, since for 
	every $a \in [0,1]$ we have $\tau(S_{l_a}) = \rho(S_{l_a}) = a^4$. \\
	\end{proof}
	The next theorem shows that the class of copulas attaining the lower bound is very small. 
	\begin{Theorem}\label{sharpla}
	 Within the class $\CSL$ the only copulas for which we have $\tau(S_\delta)=\rho(S_\delta)$ are the copulas of the 
	 form $S_{l_a}$ according to Example \ref{ex:deltaab}.
	\end{Theorem}
     \begin{proof}
      According to the proof of the previous theorem, the condition $\tau(\Sd) = \rho(\Sd)$ is equivalent to having  
	\begin{align*}
		\Sd(x,y)+ K_{\Sd}(x,[0,y])( K_{\Sd}(y,[0,x])- x)- yK_{\Sd}(y,[0,x])=0
	\end{align*}
	for $\lambda_2$-almost all $(x,y) \in [0,1]^2$. As a direct consequence, for $\lambda_2$-almost all 
	$(x,y)$ with $y < x$ 
	\begin{align}\label{gl:kernsharp}
		 K_{\Sd}(x,[0,y])\left(\tfrac{\delta(x)-x^2}{x}\right) = 0
	\end{align}
	had to hold. Set $b := \inf \left\lbrace x\in (0,1) : \delta(x) = x^2\right\rbrace$.
    If $b=0$ then $\delta = \delta_\Pi$ as well as $S_\delta=\Pi =S_{l_0}$ follows. 
    If $b > 0$ then for $\lambda_2$-almost all $(x,y) \in (0,b)^2$ with $y<x$ we have  
	\begin{align*}
		\frac{y}{x}w_\delta(x)-\frac{y}{x^2}\delta(x) = 0,
	\end{align*}
	which is equivalent to the condition that
	\begin{align*}
			w_\delta(x) = \frac{\delta(x)}{x}
	\end{align*}
	holds for $\lambda$-almost $x \in (0,b)$.  
	Using Lipschitz continuity of $\delta$ and solving this first order differential equation 
	with the boundary condition $\delta(b)=b^2$ directly yields $$ \delta(x) = b\,x$$
	for $x \in [0,b]$, which completes the proof since we have shown $\delta=l_b$. 
\end{proof}     	
	
	\begin{Theorem}\label{thm:final}
		The set $\Omega^{LSL}$ is convex and compact.		
	\end{Theorem}
	\begin{proof} 
	   Simplifying notation we will write $\tau(\delta) := \tau(S_{\delta})$ 
		and $\rho(\delta) := \rho(S_{\delta})$ for every $\delta \in \DSL$ throughout 
		the rest of the proof.
		Continuity of concordance measures and the fact that continuous images of 
		compact sets are compact imply that
		$\Omega^{LSL}$ is compact. It therefore remains to show convexity, which can be done 
		as follows:
		Consider two points $(\tau(\delta_1), \rho(\delta_1))$, $(\tau(\delta_2), 
		\rho(\delta_2))$ in $\Omega^{LSL}$ and, without loss of generality assume $\delta_1 \neq \delta_2$. 
		We want to show the existence of some $\delta \in \DSL$ fulfilling
        \begin{equation}\label{glconvex}
        (\tau(\delta), \rho(\delta)) = \tfrac{1}{2}\left(\tau(\delta_1), 
        \rho(\delta_1)  \right) \, + 
         \tfrac{1}{2}\left(\tau(\delta_2), \rho(\delta_2)  \right).
        \end{equation}
        Notice that finding such a $\delta$ is trivial if either $\rho(\delta_1)=\rho(\delta_2)$ or 
        $\tau(\delta_1)=\rho(\delta_1)$ and $\tau(\delta_2)=\rho(\delta_2)$ holds, since in this case
        a convex combination of $\delta_1$ and $\delta_2$ will do. In what follows we will therefore assume  
        that none of these two conditions holds.  
		Convexity of $\DSL$ implies that $\delta_3 := \tfrac{1}{2}(\delta_1 +\delta_2)$ is an element of $\DSL$. 
		Moreover we obviously have 
		$$
		\left(\tau(\delta_3), \rho(\delta_3) \right) = \left(\tau(\delta_3), 
		\tfrac{1}{2}(\rho(\delta_1)+\rho(\delta_2)) \right),
		$$
        and Lemma \ref{strictconvex} implies  
        $\tau(\delta_3) < \tfrac{1}{2}(\tau(\delta_1)+ \tau(\delta_2))$.
        For $\alpha,a \in \unit$ define the function $h_{\alpha,a}: [0,1] \rightarrow [0,1]$ by 
		$$
		h_{\alpha,a}(t):= (1-\alpha)\delta_3(t) + \alpha l_a(t).
		$$ 
        Then obviously $h_{\alpha,a} \in \DSL$, and for every pair of sequences 
        $(\alpha_n)_{n \in \mathbb{N}},(a_n)_{n \in \mathbb{N}}$ in $[0,1]$ 
        converging to $\alpha \in [0,1]$ and $a\in [0,1]$, respectively, we have that $(h_{\alpha_n,a_n})_{n \in \mathbb{N}}$
        converges uniformly to $h_{\alpha,a}$. 
        As a consequence, the mapping $\iota: [0,1]^2 \rightarrow [0,1]^2$, defined by
        $$
        \iota(\alpha,a)=\left(\tau(h_{\alpha,a} ),\rho(h_{\alpha,a} ) \right)
        $$ 
        is continuous. For every $a \in [0,1]$ defining $\Gamma_a$ by
        $$
        \Gamma_a:=\left\{\left(\tau(h_{\alpha,a} ),\rho(h_{\alpha,a})\right): \, \alpha \in [0,1] \right\},
        $$ 
        it therefore follows that $\Gamma_a$ is compact and connected, and that $\Gamma_a$ contains the points 
        $(\tau(\delta_3),\rho(\delta_3))$ and $(\tau(l_a),\rho(l_a))$ (see Figure \ref{fig:thmfinal} for 
        an illustration). Continuity in $a$ implies the existence of some $a_0 \in [0,1]$ fulfilling 
        $$
       \tfrac{1}{2}\left(\tau(\delta_1), 
        \rho(\delta_1)  \right) \, + 
         \tfrac{1}{2}\left(\tau(\delta_2), \rho(\delta_2)  \right) \in \Gamma_{a_0}.
        $$ 
        Having that, by construction of $\Gamma_a$ there exists some $\alpha_0$ with  
        $$
        \left(\tau(h_{\alpha_0,a_0} ),\rho(h_{\alpha_0,a_0})\right) = 
        \tfrac{1}{2}\left(\tau(\delta_1), 
        \rho(\delta_1)  \right) \, + 
         \tfrac{1}{2}\left(\tau(\delta_2), \rho(\delta_2)  \right).
         $$
         In other words, the diagonal 
        $h_{\alpha_0,a_0} \in \DSL$ fulfills equation (\ref{glconvex}), which completes the proof. 
       	\end{proof}
	
\begin{figure}[ht!]
		\centering
		\includegraphics[scale=0.54]{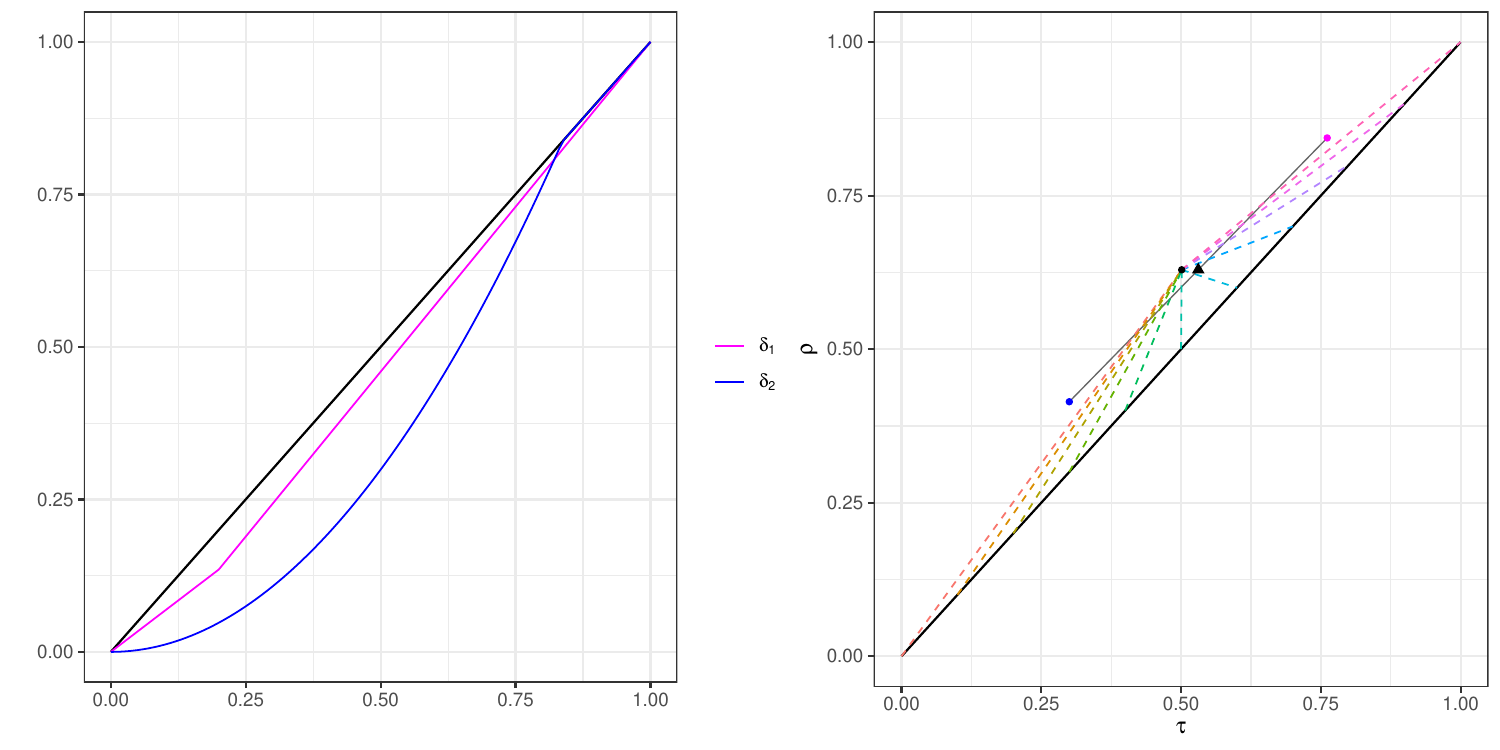}
		\caption{Illustration of the construction used in the proof of Theorem \ref{thm:final}.
		The magenta and the blue points denote $\left(\tau(\delta_1),\rho(\delta_1)  \right)$ and 
		$\left(\tau(\delta_2),\rho(\delta_2)  \right)$, respectively, the 
		black triangle their arithmetic mean. 
		The black point denotes $\left(\tau(\delta_3), \rho(\delta_3) \right)$, the 
		 dashed lines originating from it the sets $\Gamma_a$ with 
		$a \in \{0,\frac{1}{10},\ldots,\frac{9}{10},1\}$.}
		\label{fig:thmfinal}
	\end{figure}

\subsection*{Acknowledgements}
\noindent The first author gratefully acknowledges the support of the EXDIGIT (Excellence in 
Digital Sciences and Interdisciplinary Technologies) project, funded by Land Salzburg under 
grant number 20204-WISS/263/6-6022. \\
The second author gratefully acknowledges the support of the WISS 2025 project 
‘IDA-lab Salzburg’ (20204-WISS/225/197-2019 and 20102-F1901166-KZP).

\bibliographystyle{plain}

\appendix
\section{ }\label{appendix}
 \noindent It is well known that Ginni's $\gamma$, Spearman's footrule $\phi$ and 
 Blomqvist's $\beta$ of a pair $(X,Y)$ of random variables only depend on the 
 underlying copula $C$ and that the following formulas hold (see \cite{NelsenCopulas}):
 	\begin{align*}
 		\gamma(C) &= 4\int\limits_{[0,1]} C(x,x) \hspace*{1mm} d\lambda(x)
 					+ 4\int\limits_{[0,1]} C(x,1-x) \hspace*{1mm} d\lambda(x) - 2\\
 		\phi(C) &= 6\int\limits_{[0,1]} C(x,x) \hspace*{1mm} d\lambda(x)-2\\
 		\beta(C) &= 4\hspace*{0.5mm} C\left(\tfrac{1}{2},\tfrac{1}{2}\right)-1
 	\end{align*}
 	For LSL copulas the afore-mentioned formulas only depend on the diagonals and the following simple lemma holds:
 	\begin{Lemma}\label{gammabeta}
 		For every lower semilinear copula $\Sd \in \CSL$ Ginni's $\gamma$ 
 		Spearman's footrule $\phi$ and Blomqvist's $\beta$ are given by
 		\begin{align}
 			\gamma(\Sd ) &= 4\int\limits_{[0,\tfrac{1}{2}]} \delta(x)+ x\tfrac{\delta(1-x)}{1-x}
 					\hspace*{1mm} d\lambda(x) + 4\int\limits_{[\tfrac{1}{2},1]} 
 					\tfrac{\delta(x)}{x} \hspace*{1mm} d\lambda(x) -2\\
 			\phi(\Sd ) &= 6\int\limits_{[0,1]} \delta(x) \hspace*{1mm} d\lambda(x)-2\\
 			\beta(\Sd ) &= 4\hspace*{0.5mm} \delta(\tfrac{1}{2})-1,
 		\end{align}
 		and fulfill $\gamma(\Sd), \phi(\Sd ), \beta(\Sd ) \in \unit$. 
 	\end{Lemma} 
 	\begin{proof}
	 The only non-obvious formula is the one concerning $\gamma$, which follows from 
	 \begin{align*}
	 	\gamma(\Sd ) &= 4\int\limits_{[0,1]} \Sd(x,x) \hspace*{1mm} d\lambda(x)
 					+ 4\int\limits_{[0,1]} \Sd(x,1-x) \hspace*{1mm} d\lambda(x) - 2\\
 				&= 4\int\limits_{[0,1]} \delta(x) \hspace*{1mm} d\lambda(x)
 				  + 4\int\limits_{[0,\frac{1}{2}]} x\tfrac{\delta(1-x)}{1-x} \hspace*{1mm} 
 				   d\lambda(x) + 4\int\limits_{[\frac{1}{2},1]} (1-x)\tfrac{\delta(x)}{x} 
 				\hspace*{1mm} d\lambda(x) - 2\\
 				&= 4\int\limits_{[0,\frac{1}{2}]} \delta(x) \hspace*{1mm} d\lambda(x)
 				  + 4\int\limits_{[0,\frac{1}{2}]} x\tfrac{\delta(1-x)}{1-x} \hspace*{1mm} 
 						d\lambda(x) + 4\int\limits_{[\frac{1}{2},1]} 
 						\tfrac{\delta(x)}{x} \hspace*{1mm} d\lambda(x) - 2\\
 				&= 4\int\limits_{[0,\frac{1}{2}]} \delta(x) + x\tfrac{\delta(1-x)}{1-x} 
 					\hspace*{1mm} d\lambda(x) + 4\int\limits_{[\frac{1}{2},1]} 
 					\tfrac{\delta(x)}{x} \hspace*{1mm} d\lambda(x) - 2.
	 \end{align*}
     \end{proof}

\end{document}